\documentclass[11pt,a4paper,twoside,english]{article}

\usepackage{amsmath}
\usepackage{amssymb}
\usepackage{amsopn}
\usepackage{amsthm}
\usepackage{tikz}
\usepackage{tikz-cd}
\usepackage{pgfplots}
\usepackage{wrapfig}
\usepackage{mathrsfs}
\usepackage[all]{xy}
\usepackage{mathtools,xparse}
\usepackage{theoremref}
\usepackage{combelow}
\usepackage{pifont}
\usepackage{pdfpages}

\usepackage{hyperref}

\newcommand{\cmark}{\ding{51}}%
\newcommand{\xmark}{\ding{55}}%
\newcommand{\diffto}{\xrightarrow{\raisebox{-0.2 em}[0pt][0pt]{\smash{\ensuremath{\sim}}}}}

\DeclareMathOperator{\N}{\mathbb{N}}
\DeclareMathOperator{\Z}{\mathbb{Z}}
\DeclareMathOperator{\R}{\mathbb{R}}

\DeclareMathOperator{\Cas}{Cas}

\DeclareMathOperator{\spl}{\mathfrak{sl}_2(\R)}

\DeclareMathOperator{\g}{\mathfrak{g}}
\DeclareMathOperator{\F}{\mathcal{F}}
\DeclareMathOperator{\Y}{\mathcal{Y}}
\DeclareMathOperator{\Yr}{\mathcal{Y}^{\mathrm{reg}}}

\DeclareMathOperator{\supp}{supp}
\DeclareMathOperator{\id}{id}

\DeclareMathOperator{\sign}{sign}

\DeclareMathOperator{\dif}{\mathrm{d}}
\DeclareMathOperator{\CC}{\mathscr{C}}
\DeclareMathOperator{\Lie}{\mathcal{L}}

\newtheorem{question}{Question}[section]
\newtheorem{theorem}{Theorem}[section]
\newtheorem*{theorem*}{Theorem}

\newtheorem{proposition}[theorem]{Proposition}
\newtheorem{lemma}[theorem]{Lemma}
\newtheorem{corollary}[theorem]{Corollary}
\newtheorem{remark}[theorem]{Remark}

\DeclareDocumentCommand{\norm}{ s m }{%
	\IfBooleanTF{#1}
	{#2}
	{\lVert{#2}\rVert}%
}
\DeclareDocumentCommand{\nl}{ s m }{%
	\IfBooleanTF{#1}
	{#2}
	{\lVert{#2}\rVert_{l}}%
}

\DeclareDocumentCommand{\fkn}{ s m }{%
	\IfBooleanTF{#1}
	{#2}
	{\lVert{#2}\rVert_{C^k (B_n)}}%
}
\DeclareDocumentCommand{\fan}{ s m }{%
	\IfBooleanTF{#1}
	{#2}
	{\lVert{#2}\rVert_{C^{2|a|} (B_n)}}%
}

\title{The Poisson cohomology of $\mathfrak{sl}_2^*(\mathbb{R})$}
\author{Ioan M\u{a}rcu\cb{t}, Florian Zeiser}
\begin{document}
\maketitle
 \begin{abstract}
  We compute the smooth Poisson cohomology of the linear Poisson structure associated with the Lie algebra $\spl$.
 \end{abstract}
\tableofcontents

\section{Introduction}

Let $M$ be a smooth manifold. The space of $C^{\infty}$-multi-vector fields on $M$:
 \begin{align*}
  \mathfrak{X}^{\bullet}(M)&:= \Gamma (\wedge ^{\bullet} TM)
 \end{align*}
carries a natural extension of the Lie bracket, called the Schouten-Nijenhuis bracket (see e.g.\ \cite{Laurent2013}), which makes $\mathfrak{X}^{\bullet}(M)$ into a graded Lie algebra:  
 \begin{equation*}
  \begin{array}{cccc}
   [\cdot,\cdot]:&\mathfrak{X}^{p+1}(M)\times \mathfrak{X}^{q+1}(M)&\to &\mathfrak{X}^{p+q+1}(M).
  \end{array}
 \end{equation*}
 
A Poisson structure on $M$ is a bivector field $\pi \in \mathfrak{X}^2(M)$ satisfying
\[[\pi,\pi]=0.\]
By the graded Jacobi identity, this equation is equivalent to 
\[\dif_{\pi}^2=0,\ \ \textrm{where}\ \ \dif_{\pi}:=[\pi,\cdot]:\mathfrak{X}^{\bullet}(M)\to \mathfrak{X}^{\bullet +1}(M).\]
The cohomology of the resulting chain complex
\[(\mathfrak{X}^{\bullet}(M), \dif_{\pi}),\]
is called the Poisson cohomology of $(M,\pi)$, and was introduced by Lichnerowicz \cite{Lich77}. The Poisson cohomology groups, denoted by 
\[H^{\bullet}(M,\pi),\] 
encode infinitesimal information about the Poisson structure. In low degrees they have the following interpretations: $H^0(M,\pi)$ consists of so-called Casimir functions, which are the ``smooth functions'' on the leaf-space; $H^1(M,\pi)$ plays the role of the Lie algebra of the ``Lie group'' of outer automorphisms of the Poisson manifold, $H^2(M,\pi)$ is the ``tangent space'' to the Poisson-moduli-space, or the space of infinitesimal deformations of the Poisson structure, and in $H^3(M,\pi)$ we can find obstructions to extending infinitesimal deformations to actual deformations. However, these interpretations are mostly of a heuristic or formal nature, since there are no general results asserting them, and their validity is poorly understood. 

Poisson cohomology is hard to compute due to the lack of general methods. The existing techniques are specialized to certain classes of Poisson structures, which we briefly outline below
\begin{itemize}
\item \underline{Mildly degenerate Poisson structures}. As noticed already in \cite{Lich77}, for symplectic structures (i.e.\ non-degenerate Poisson) the Poisson complex is isomorphic to the de Rham complex, and so it computes the usual (real) cohomology of the manifold; this is also the only case when the Poisson differential is elliptic. Similar techniques apply also to Poisson structures which are almost everywhere non-degenerate and have ``mild'' singularities. For these one can use singular de Rham forms. This was first worked out in dimension 2: for linear singularities in \cite{Radko}, for quadratic singularities in \cite{Nakanishi_97}, and for general singularities in \cite{Monnier}, and in general dimension: for log-symplectic structures in \cite{MO14,GMP,Lanius_1} and for higher order singularities in \cite{Lanius_2}.
\item \underline{Regular Poisson structures} have a non-singular symplectic foliation, which induces a filtration on the Poisson complex; the first pages of the resulting spectral sequence are described in terms of foliated cohomology \cite{Vaisman}. For simple foliations, this technique can be used to obtain explicit results as in \cite{Xu92}, or as in \cite{Gammella} where the Poisson cohomology of the regular part of certain duals of low-dimensional Lie algebras is calculated. 
\item \underline{Compact-type}. For the linear Poisson structure on the dual of a compact semi-simple Lie algebra, Conn showed that the Poisson cohomology vanishes in first and second degree, and he used this in the proof of the linearization theorem \cite{Conn85}.
The full Poisson cohomology associated to compact Lie algebras was calculated in \cite{GW92}. The proof therein uses averaging over the fibers of a source compact Lie groupoid. This technique has been extended to ``compact-type'' Poisson manifolds, already in \cite{Xu92} for simple, regular foliations, and more recently in \cite{PMCT} for Poisson manifolds which admit a source-proper Lie groupoid integrating them.
\item \underline{Other categories}. There are several calculations of Poisson cohomology in categories different from $C^{\infty}$, such as: formal, analytic, holomorphic, or algebraic Poisson cohomology. We will not discuss these results here, because the techniques involved are usually quite different and rarely of use in the $C^{\infty}$-setting.
\end{itemize}

In this paper we calculate the Poisson cohomology of the linear Poisson structure on the dual of the Lie algebra $\spl$ (see Theorem \ref{main theorem}):
\[(\mathfrak{sl}_2^*(\mathbb{R}),\pi).\]
Our interest in calculating this cohomology originates in our study of the local structures of ``generic Poisson structures'' in odd-dimension, which transversely to the singular leaves are linearly approximated by $\mathfrak{sl}_2^*(\mathbb{R})$ or $\mathfrak{so}_3^*(\R)$; the second case being well understood \cite{Conn85}. There are several other reasons to consider specifically this example. First, $\spl$ does not fit into any of the classes above for which techniques are known, and therefore its calculation requires some new insights and ideas. Secondly, semisimple Lie algebra have been considered many times in the Poisson framework, especially in relation to the problem of linearization \cite{Wein83,Conn84,Conn85,Wein87,GW92}. Finally, the Poisson cohomology of $(\mathfrak{sl}_2^*(\mathbb{R}),\pi)$ has a representation theoretic flavor, as it is isomorphic to the Chevalley-Eilenberg cohomology of $\spl$ with coefficients in the infinite-dimensional representation $C^{\infty}(\mathfrak{sl}_2^*(\mathbb{R}))$. 

As we will see in Section \ref{section: geometric interpretation}, all classes in $H^{\bullet}(\mathfrak{sl}_2^*(\mathbb{R}),\pi)$ have a clear geometric meaning,
and therefore the construction of representatives is quite intuitive. The difficult part, which will occupy most of the paper, is proving that the elements we construct cover all cohomology classes. This is also reflected in the literature, as, in one way or another, representatives for all classes have appeared in various contexts. In \cite[Prop 6.3]{Wein83}, Weinstein constructed a non-analytic deformation of $(\mathfrak{sl}_2^*(\mathbb{R}),\pi)$ which is non-linearizable. In fact, by our main result, all infinitesimal deformations in $H^2(\mathfrak{sl}_2^*(\mathbb{R}),\pi)$ are obtained by a similar procedure. Also with the aim of constructing deformations of semi-simple Lie algebras, the results in \cite{Wein87} and in \cite[Thm 4.3.9]{Zung_Book} yield non-trivial classes in $H^1(\mathfrak{sl}_2^*(\mathbb{R}),\pi)$, and to some extend, these classes have appeared also in \cite{Nakanishi_91}. Let us mention also that the Poisson cohomology of the regular part of $\mathfrak{sl}_2^*(\mathbb{R})$ was considered in \cite{Gammella}, where it is proven to be infinite dimensional. 
 
The formal and polynomial Poisson cohomology of $\spl$ can be calculated using standard representation theory (see e.g.\ \cite{Laurent2013,Pichereau, Nakanishi_91}), and, most likely, the analytic Poisson cohomology can be deduced using methods from  \cite{Conn84}. 

The paper is structured as follows. In Section \ref{section results} we state our main result, and build representatives for all Poisson cohomology classes. In Section \ref{section: geometric interpretation} we discuss the algebra of Casimir functions, we calculate the induced Schouten-Nijenhuis bracket in cohomology, we construct groups of outer automorphisms and deformations, and we study the action of outer-automorphisms on deformations. In Section \ref{section: formal}, by using the formal Poisson cohomology, we reduce the problem to that of calculating flat Poisson cohomology. In Section \ref{section: flat PC} we introduce the flat foliated complex, which, as in the regular case, can be used to compute flat Poisson cohomology. In Section \ref{section flat foli} we calculate the flat foliated cohomology. For this, we construct a retraction of the foliation to a subset, which represents the ``cohomological skeleton'' of the foliation, and show that the retraction is a homotopy equivalence. The retraction is built as the infinite flow of a vector field; the analysis of the flow of this vector field is left for Section \ref{section: analysis}.

Summarizing the main steps of the proof, we extract a general strategy for calculating Poisson cohomology (a version of this scheme was used in  \cite{Ginz96} for a specific 2-dimensional Poisson structure):
\begin{itemize}
\item[1.] Calculate the formal Poisson cohomology at the singularity of the foliation, and reduce the problem to calculating the cohomology of the ``flat Poisson complex'';
\item[2.] Treat this subcomplex as it came from a regular Poisson structure, and reduce the problem to calculating flat foliated cohomology;
\item[3.] For calculating (flat) foliated cohomology, try to build a contraction to a ``cohomological skeleton''.
\end{itemize}

In future work, we will attempt to use this strategy for other Poisson structures. In particular, we will continue the study of the Poisson cohomology associated to other semi-simple Lie algebras. 

 \section{The main result}\label{section results}

To state the results we identify $\mathfrak{sl}_2^*(\mathbb{R})$ with $\R^3$ in such a way that the linear Poisson structure is given by
 \begin{align*}
  \pi &:=x\partial _y \wedge \partial _z +y\partial_z \wedge \partial_x -z\partial_x\wedge \partial_y
 \end{align*}

The symplectic foliation of $\pi$ can be described with using the following basic Casimir function, which will be used throughout the paper:
 \begin{align}\label{Casimir}
  f(x,y,z)&:= x^2+y^2-z^2.
 \end{align}
 \begin{wrapfigure}{r}{4cm}
  \vspace{-20pt}
 \includegraphics[scale=0.3]{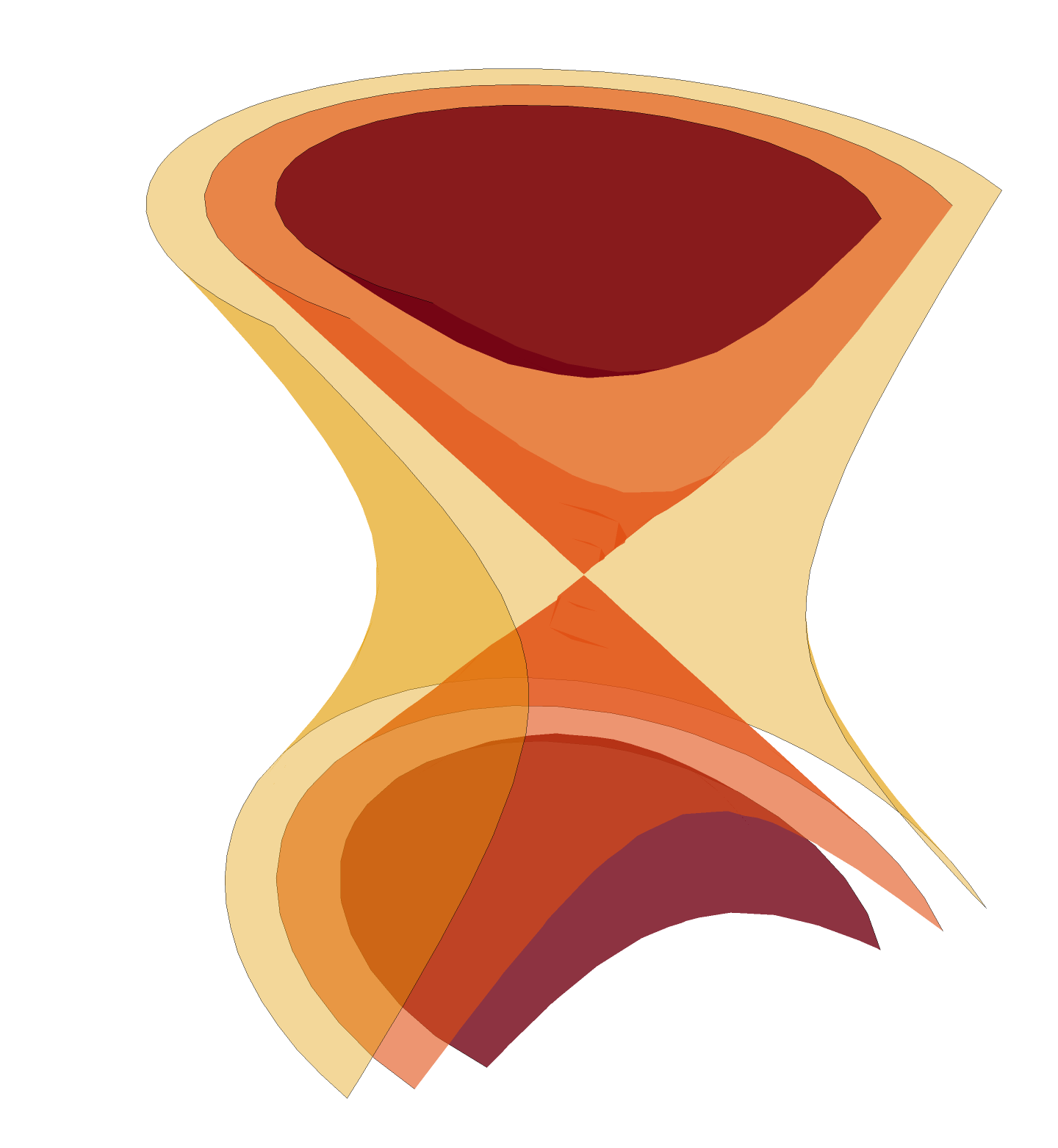}
 \caption{The level sets of $f$}
 \vspace{30pt}
 \end{wrapfigure}
The symplectic leaves are the following families of submanifolds: the one-sheeted hyperboloids
 \[S_{\lambda}:=f^{-1}(\lambda), \ \ \lambda>0;\]
the two sheets of the hyperboloids
\[S^{\pm}_{\lambda}:=f^{-1}(\lambda)\cap\{\pm z>0\}, \ \  \lambda<0;\]
and the cone $f^{-1}(0)$ decomposes into three leaves; 
\[S^{\pm}_{0}:=f^{-1}(0)\cap\{\pm z>0\}, \ \ S_0:=\{0\}.\]

The leaf-space, denoted by $\Y$, is obtained by identifying points belonging to the same leaf, and taking the quotient topology. The regular part of $\Y$:
\[\mathcal{Y}^{\mathrm{reg}}:=\Y\backslash\{S_0\}\]
is a smooth 1-dimensional non-Hausdorff manifold. Its smooth structure is determined by the quotient map being a submersion. Explicitly, a smooth atlas is given by: 
\[\big\{(U^+,\varphi^+),(U^{-},\varphi^-)\big\}\]  
\[ U^{\pm}=\big\{S^{\pm}_{\lambda} \ | \ \lambda\leq 0\big\}\cup \big\{S_{\lambda}\ | \ \lambda>0\big\}\subset \Yr,\]
\[\varphi^{\pm}:U^{\pm}\diffto  \R,\ \ \ \ S^{\pm}_{\lambda}\mapsto \lambda, \ \ 
 S_{\lambda}\mapsto \lambda.\]

\begin{wrapfigure}{r}{4.5cm}
 \vspace{20pt}
  \begin{tikzpicture}
  \draw (-1.5,0) -- (0,0);
  \draw (-3,0.2) -- (-1.5,0.2);
  \draw (-3,-0.2) -- (-1.5,-0.2);
  \draw[fill=black] (-1.5,0.2) circle (0.2ex);
  \draw[fill=white] (-1.5,0) circle (0.2ex);
  \draw[fill=black] (-1.5,-0.2) circle (0.2ex);
  \end{tikzpicture}
  \caption{$\Yr$}
 \end{wrapfigure}
This atlas allows us to identify $\Yr$ with two copies of $\R$ glued along $(0,\infty)$:
\[\Yr\simeq \R\sqcup_{(0,\infty)}\R.\]
The algebra of smooth functions on $\Yr$ is \[C^{\infty}(\Yr)=\{(h_1,h_2)\in C^{\infty}(\R)\times C^{\infty}(\R)\ | \ h_1|_{(0,\infty)}=h_2|_{(0,\infty)}\}.\]

The 0-th Poisson cohomology group consists of smooth function constant along the symplectic leaves, also called Casimir functions, denoted by:
\[\Cas(\mathfrak{sl}_2^*(\mathbb{R})):=H^0(\mathfrak{sl}_2^*(\mathbb{R}),\pi).\] 
In Subsection \ref{sub: Casimirs}, we will prove the following:
\begin{proposition}\label{prop:Casimirs}
The algebra of Casimir functions is isomorphic to the algebra of smooth functions on the regular part of the leaf-space: 
\begin{equation}\label{eq: functions on the orbit space}
C^{\infty}(\Yr)\simeq \Cas(\mathfrak{sl}_2^*(\mathbb{R})),
\end{equation}
and the isomorphism is given:
\[h=(h_1,h_2)\mapsto \widetilde{h},\]
\[
\widetilde{h}(x,y,z):=
\begin{cases} h_1 (f(x,y,z)), & z\geq 0 \\
h_2(f(x,y,z)),  & z < 0
\end{cases}.
\]
\end{proposition}

Denote the singular cone, its ``outside'' and its ``inside'', respectively, by
\[Z:=\{f=0\},\ \ \  O:=\{f>0\}, \ \ \ I:=\{f<0\}.\]
We introduce two Poisson vector fields $T$ and $N$ on $O\cup I$:
\begin{align*}
T|_O&:=\partial_z+\frac{z}{x^2+y^2}(x\partial_x+y\partial y),\ \ \ \ \  \ T|_I:=0,\\
N|_O&:=\frac{1}{2(x^2+y^2)}(x\partial_x+y\partial_y), \ \ \ \ \ \ N|_I:=\frac{1}{2(y^2-z^2)}(y\partial_y + z\partial_z).
\end{align*}
These formulas come from the special coordinate systems:
\begin{align}\label{coordinates O}
&\textrm{on } O: &\theta=\mathrm{tan}^{-1}\left(\frac{y}{x}\right),\ \  \ \ \ \ \ \  &w=z, &f=x^2+y^2-z^2,\\
&\textrm{on }I:\label{coordinates I} &\xi=\mathrm{tanh}^{-1}\left(\frac{y}{z}\right),\ \ \ \ \ \ \ \  &v=x, &f=x^2+y^2-z^2, 
\end{align}
in which:
\begin{equation}\label{coordinates pi O}
\pi|_O=\partial_{\theta}\wedge\partial_w, \ \ \ \ \   T|_O=\partial_w,\ \ \ \ \   N|_O=\partial_f,
\end{equation}
\begin{equation}\label{coordinates pi I}
\pi|_I=\partial_{\xi}\wedge\partial_v,\ \ \ \ \ \  T|_I=0,   \ \ \ \ \ \ \  N|_I=\partial_f.
\end{equation}
We will use the collection of flat Casimir functions:
\[C^{flat}=\big\{\chi\in \Cas(\mathfrak{sl}_2^*(\mathbb{R}))\ \  |\ \   \eta \textrm{ vanishes flatly at } 0 \big\}.\]
Proposition \ref{prop:Casimirs} implies that any $\chi \in C^{flat}$ vanishes flatly along the entire cone $Z$ (see Subsection \ref{sub: Casimirs}). Since the singularities of $T$ and $N$ along $Z$ are given by rational functions, if follows that, for all $\chi\in C^{flat}$, 
\[\chi T, \ \ \chi N,\]
extend to smooth vector fields on $\R^3$ that vanish flatly on $Z$. We will also use the collection of Casimir functions with support outside of the cone:
\[C^{out}=\big\{\eta \in \Cas(\mathfrak{sl}_2^*(\mathbb{R}))\ \  |\ \ \supp(\eta) \subset \overline{O}\big\}.\]

We state now the main result of the paper:

\begin{theorem}\label{main theorem}
The Poisson cohomology of $(\mathfrak{sl}_2^*(\mathbb{R}),\pi)$ is given by:
\begin{itemize}
\item Every class in $H^1(\mathfrak{sl}_2^*(\mathbb{R}),\pi)$ can be represented as
  \[ \chi N + \eta T\]
for unique functions $\chi \in C^{flat}$ and $\eta \in C^{out}$.
\item Every class in $H^2(\mathfrak{sl}_2^*(\mathbb{R}),\pi)$ can be represented as
\begin{align*}
  \eta N\wedge T
  \end{align*}
  for a unique function $\eta \in C^{out}$.
  \item For the third Poisson cohomology group we have
   \begin{align*}
   H^3(\mathfrak{sl}_2^*(\mathbb{R}),\pi)&\simeq \R[[f]]\ \partial_x\wedge\partial_y\wedge\partial_z,
   \end{align*}
   where $\R[[f]]$ denotes the ring of formal power series in $f$.
\end{itemize}
\end{theorem}


\section{Geometric interpretation}\label{section: geometric interpretation}

In this section we prove Proposition \ref{prop:Casimirs} and we explore the geometric meaning of our calculation of the Poisson cohomology of $(\mathfrak{sl}_2^*(\mathbb{R}),\pi)$. In particular, we calculate the Schouten-Nijenhuis bracket in cohomology, we describe groups of Poisson-diffeomorphisms ``integrating'' the first Poisson cohomology Lie algebra, we build deformations corresponding to the second Poisson cohomology, and describe some identifications between the deformations. The entire discussion leaves many open questions, which hopefully will be answered in the future.

\subsection{The algebra of Casimir functions}\label{sub: Casimirs}

\begin{wrapfigure}{l}{3cm}
\vspace{-5pt}
 \includegraphics[scale=0.2]{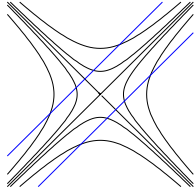}
 \caption{$\gamma_1$ and $\gamma_2$ intersecting the level sets of $f$ in the $y$-$z$-plane}
 \vspace{-30pt}
 \end{wrapfigure}
 
We begin with 
\begin{proof}[Proof of Proposition \ref{prop:Casimirs}]
First, we show that the map is indeed defined, i.e.\ for any  $h=(h_1,h_2)\in C^{\infty}(\Yr)$, the function $\widetilde{h}$ is indeed smooth. By subtracting $h_2\circ f$ from $\widetilde{h}$, we may assume that $h_2=0$. So let $h_1\in C^{\infty}(\R)$, with $h_1|_{(0,\infty)}=0$, and note that: 
\[
\widetilde{h}(x,y,z)=
\begin{cases} h_1 (f(x,y,z)), & z\geq 0\  \\
0,  & z < 0
\end{cases}.
\]
On $z\geq 0$, $h_1(f(x,y,z))$ is smooth and it vanishes on $f\geq 0$; in particular it vanishes flatly on the plane $z=0$. Therefore its extension by $0$ on $z<0$ is a smooth function on $\R^3$.

Next, we show that any Casimir function comes from an element in $C^{\infty}(\Yr)$. For this, we define two straight lines $\gamma_{1},\gamma_{2}:\R\to \mathfrak{sl}_2^*(\mathbb{R})$:
\begin{equation}\label{eq:curves}
\gamma_{1}(t)=\frac{1}{2}\big(0,-t-1,1-t\big), \ \ 
\gamma_{2}(t)=-\gamma_{1}(t).
\end{equation}
Both lines are transverse to the leaves of the foliation and satisfy:
\begin{equation}\label{eq:transversals}
f\circ \gamma_{1}(t)=t,  \ \ \ \ \ f\circ \gamma_{2}(t)=t.
\end{equation}
Both lines cut leaves at most once; their intersections are indicated below:
\begin{center}
\begin{tabular}{||c||c|c|c|c|}
\hline
& $S_0$ 
& $S_{\lambda}$, 
$\lambda>0$ & 
$S^{-}_{\lambda}$, 
$\lambda\leq 0$ & 
$S^{+}_{\lambda}$,\  
$\lambda\leq 0$ \\ 
\hline \hline
$\gamma_1$ & \xmark & \cmark & \xmark & \cmark\\ 
\hline
$\gamma_2$ & \xmark   & \cmark & \cmark & \xmark\\
\hline
\end{tabular}
\end{center}

Consider a Casimir function $g\in \Cas(\mathfrak{sl}_2^*(\mathbb{R}))$, and denote by
\[ h_1:=g\circ \gamma_1\in C^{\infty}(\R), \ \ h_2:=g\circ \gamma_2\in C^{\infty}(\R).\]
Since for $t>0$, $\gamma_1(t),\gamma_2(t)\in S_t$, it follows that $
h_1(t)= h_2(t)$, and so 
\[h=(h_1,h_2)\in C^{\infty}(\Yr).\]
We have that $g=\widetilde{h}$. This follows because both are Casimir functions, their compositions with the $\gamma_1$ and $\gamma_2$, respectively, yield the same result, and the two lines cut all regular leaves.
\end{proof}

Next, let us note that under the isomorphism \eqref{eq: functions on the orbit space}, we have that 
\[C^{flat}\simeq C^{\infty}_0(\Yr),\]
where $C^{\infty}_0(\Yr)$ consists of pairs $(h_1,h_2)\in C^{\infty}(\Yr)$ with the property that $h_1$ and $h_2$ vanish flatly at their $0$s, respectively. This follows by comparing the Taylor series at $0$ of $h_{1}(t^2)=\widetilde{h}(t,0,0)$; and similarly for $h_2$. Hence Casimirs which vanish flatly at the origin, actually vanish flatly along $Z$.

Note also that, under the isomorphism \eqref{eq: functions on the orbit space}, we have that 
\[C^{out}\simeq C^{\infty}_{\to}(\Yr),\]
where $C^{\infty}_{\to}(\Yr)$ consists of pairs $(h_1,h_2)\in C^{\infty}(\Yr)$ satisfyng: $h_1=h_2$ and $\supp(h_1)\subset [0,\infty)$. 


\subsection{The Schouten-Nijenhuis bracket}
The Schouten-Nijenhuis bracket on multi-vector fields descends to a bracket on Poisson cohomology, which can be easily calculated for $(\mathfrak{sl}_2^*(\mathbb{R}),\pi)$.

First, note that $T$ is tangent to the symplectic foliation, therefore
\[\Lie_T(g)=0, \ \ \textrm{for all}\ \ g\in\Cas(\mathfrak{sl}_2^*(\mathbb{R})).\]
This implies that 
\[\mathfrak{g}_T:=\{\eta T\ | \ \eta\in C^{out}\}\]
is an abelian subalgebra.

Next, note that $N$ is transverse to the symplectic foliation on $\R^3 \setminus Z$, with 
\[\Lie_N(f|_{R^3\setminus Z})=1.\]
Moreover, for $g\in \Cas(\mathfrak{sl}_2^*(\mathbb{R}))$ we have that $\Lie_N(g)$ extends to a smooth Casimir on $\R^3$. This follows because locally $N=\partial_f$, and so if $g$ corresponds to the pair $h=(h_1,h_2)\in C^{\infty}(\Yr)$, then $\Lie_N(g)$ corresponds to the pair $\partial_t h=(\partial_t h_1,\partial_t h_2)\in C^{\infty}(\Yr)$. Thus, although $N$ is only smooth on the set $\R^3 \setminus Z$, its Lie derivative $\Lie_N$ induces a derivation of the algebra of Casimir functions, denoted by
\[\partial_f:\Cas(\mathfrak{sl}_2^*(\mathbb{R}))\to \Cas(\mathfrak{sl}_2^*(\mathbb{R})),\]
which corresponds under the isomorphism \eqref{eq: functions on the orbit space} to $\partial_t$. We obtain that 
\[\mathfrak{g}_N:=\{\chi N \ | \ \chi\in C^{flat}\}\]
is a Lie subalgebra, which is isomorphic to the Lie algebra of vector fields on $\Yr$ which are flat at the origin(s):
\begin{equation}\label{g_N}
(\mathfrak{g}_N,[\cdot,\cdot])\simeq (\mathfrak{X}_0^1(\Yr),[\cdot,\cdot]).
\end{equation}

In the coordinates \eqref{coordinates pi O}, it is obvious that on $O\cup I$ 
\[[N,T]=0.\]
Using the Leibniz rule, this allows us to calculate other brackets, for example:
\begin{equation}\label{Schouten}
[\chi N, \eta N\wedge T]=(\chi\partial_f(\eta)-\eta\partial_f(\chi))N\wedge T.
\end{equation}

\begin{remark}
It is somehow surprising that the representatives we found for Poisson cohomology in degree $\leq 2$ are closed under the Schouten-Nijenhuis bracket.
\end{remark}
 Since all 3-vector fields that are flat at 0 are trivial in cohomology, we obtain the following:

\begin{corollary}\label{bracket}
The bracket induced from the Schouten-Nijenhuis bracket on Poisson cohomology
\begin{equation*}
\begin{array}{cccc}
[\cdot,\cdot]:&H^p(\mathfrak{sl}_2^*(\mathbb{R}),\pi)\times H^q(\mathfrak{sl}_2^*(\mathbb{R}),\pi)&\to&H^{p+q-1}(\mathfrak{sl}_2^*(\mathbb{R}),\pi)
\end{array}
\end{equation*}
is non-zero only for $p+q\leq 3$, and in these degrees it is 
determined by the Leibniz identity and the following relations: 
\[[N, g]=\partial_f(g), \ \ \  \ [T, g]=0,\ \  \ \ [ N, T]=0,\]
for all $g\in \Cas(\mathfrak{sl}_2^*(\mathbb{R}))$.
\end{corollary}

In particular, note that $\mathfrak{g}_N$ and $\mathfrak{g}_T$ span a Lie subalgebra, which is a semi-direct product
$\mathfrak{g}_T\rtimes \mathfrak{g}_N$, because:
\[[\chi N, \eta T]=\chi\partial_f(\eta) T.\]

\subsection{Poisson-diffeomorphisms}

Denote the Lie algebra of Poisson vector fields by:
\[\mathfrak{poiss}:=\{X\in \mathfrak{X}(\mathfrak{sl}_2^*(\mathbb{R}))\ | \ \Lie_X\pi=0 \},\]
and the ideal of Hamiltonian vector fields by:
\[\mathfrak{ham}:=\{X_{h}:=\pi^{\sharp}(\dif h)\ | \ h\in C^{\infty}(\mathfrak{sl}_2^*(\mathbb{R}))\}.\]
The quotient Lie algebra is the first Poisson cohomology:
\[H^1(\mathfrak{sl}_2^*(\mathbb{R}),\pi)=\mathfrak{poiss}/\mathfrak{ham}\simeq \mathfrak{g}_T\rtimes \mathfrak{g}_N.\]
Note that $\mathfrak{poiss}$ has a 3-term filtration by ideals:
\[0\, \unlhd\, \mathfrak{ham}\, \unlhd\, \mathfrak{ham}\rtimes \mathfrak{g}_{T}
\, \unlhd\, (\mathfrak{ham}\rtimes \mathfrak{g}_{T})\rtimes \mathfrak{g}_N=\mathfrak{poiss},\]
and $\mathfrak{ham}\rtimes \mathfrak{g}_{T}$ consists of the Poisson vector fields tangent to the foliation.

Next, we describe groups corresponding to these Lie algebras. It would be interesting to understand to what extend these groups are smooth or integrate the Lie algebras. Denote the Poisson-diffeomorphism group by
\[\mathrm{Poiss}:=\{\varphi\in \mathrm{Diff}(\mathfrak{sl}_2^*(\mathbb{R})), \ \ \varphi_*(\pi)=\pi\},\]
and the (normal) Hamiltonian subgroup by: 
\[\mathrm{Ham}\, \unlhd\, \mathrm{Poiss}.\]
The group $\mathrm{Ham}$ consists of diffeomorphisms $\varphi$ that can be connected to the identity by a smooth family of diffeomorphisms
\[\{\varphi_t\}_{t\in[0,1]}, \ \ \ \varphi_0=\mathrm{id}, \ \ \ \varphi_1=\varphi\]
that is generated by a smooth family of Hamiltonians $\{h_t\in C^{\infty}(\mathfrak{sl}_2^*(\mathbb{R}))\}_{t\in [0,1]}$:
\[\varphi'_t=X_{h_t}\circ \varphi_t, \ \ \ X_{h_t}=\pi^{\sharp}(\dif h_t).\]

Next, we associate to $\mathfrak{g}_{T}$ an abelian group:
\[G_T:=\big\{\, \mathrm{exp}(\eta T):=\textrm{time-one flow of  }\eta T\ | \ \eta \in C^{out}\big\}.\]
To see that the vector fields $\eta T$ are indeed complete, and that $G_T$ is indeed an abelian group, we use the coordinates from \eqref{coordinates pi O} $(\theta,w,f)\in S^1\times \R\times (0,\infty)$ on $O=\{f>0\}$, in which: 
\[\pi|_O=\partial_{\theta}\wedge\partial_{w}, \ \ T|_O=\partial_{w}.\]
Then the flow of $\eta T$, with $\eta=h\circ f\in C^{out}$, and $\mathrm{supp}(h)\subset [0,\infty)$, is 
\[\exp(t\eta T)(\theta,w,f)=(\theta,w+t\, h(f),f),\]
in particular, it is defined for all $t\in \R$. Because it vanishes on $\overline{I}$, $\eta T$ is complete. Note that the flow preserves the leaves. The leaves in $O$ are sent by the chart symplectomorphically to the cotangent bundle of the circle: 
\[S_{\lambda}\simeq T^*S^1.\]
Under this identification, $\exp(\eta T)$ acts by translation with $h(\lambda)\dif \theta\in \Omega^1(S^1)$. 

By the formula for the flow, the exponential is a group isomorphism:
\[\exp:\mathfrak{g}_T\diffto G_T, \ \ \exp(\eta_1T+\eta_2T)=\exp(\eta_1T)\circ \exp(\eta_2T).\]
The subgroup corresponding to $\mathfrak{ham}\rtimes \mathfrak{g}_T$ is the semi-direct product:
\[\mathrm{Ham}\rtimes G_T.\]
It would be interesting to know whether $\mathrm{Ham}\rtimes G_T$ can be characterized geometrically by the following property:
\begin{question}
Does a Poisson diffeomorphism that sends each leaf to itself belong to 
$\mathrm{Ham}\rtimes G_T$? 
\end{question}

Recall that $\mathfrak{g}_N$ is isomorphic to the Lie algebra of vector fields on $\Yr$ that are flat at the origin(s) \eqref{g_N}. Next, we build a group $G_N$ corresponding to $\mathfrak{g}_N$, which will be isomorphic to the group of diffeomorphisms of $\Yr$ which are flat at the origin(s). First consider the group
\[\mathrm{Diff}^0_0(\Yr)\subset \mathrm{Diff}(\Yr)\]
consisting of pairs $(\phi_1,\phi_2)$ of diffeomorphisms of $\R$ which fix the origin up to infinite jet (i.e.\ $\phi_i-\mathrm{id}_{\R}$ vanishes flatly at $0$), and such that $\phi_1|_{(0,\infty)}=\phi_2|_{(0,\infty)}$. For $\phi=(\phi_1,\phi_2)\in \mathrm{Diff}^0_0(\Yr)$, we build an element $\widetilde{\phi}\in G_N$. In the chart \eqref{coordinates O} on $O$, define:
\[\widetilde{\phi}(\theta,w,f)=(\theta,w,\phi_1(f))=(\theta,w,\phi_2(f)),\]
in the chart \eqref{coordinates I} on $I\cap \{z> 0\}$, define:
\[\widetilde{\phi}(\xi,v,f)=(\xi,v,\phi_1(f)),\]
and in the chart \eqref{coordinates I} on $I\cap \{z< 0\}$, define:
\[\widetilde{\phi}(\xi,v,f)=(\xi,v,\phi_2(f)).\]

Note that these three expressions extend to $Z=\{f=0\}$ as the identity map, and they coincide along $Z$ with the identity up to infinite jet. Therefore $\widetilde{\phi}$ is indeed smooth (one can also transform $\widetilde{\phi}$ to usual coordinates to check this). The local expression of $\pi$ in the charts \eqref{coordinates pi O} and \eqref{coordinates pi I}, implies that $\widetilde{\phi}$ is a Poisson diffeomorphism. Let $G_N^{0}$ be the collection of all $\widetilde{\phi}\in \mathrm{Poiss}$, with $\phi\in \mathrm{Diff}_{0}^0(\Yr)$. Since $\widetilde{\phi}$ induces $\phi$ on the regular leaf-space, we have that:
\[G_N^0\simeq \mathrm{Diff}_{0}^0(\Yr).\]

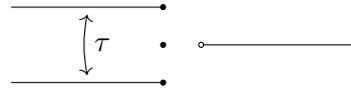
\begin{wrapfigure}{r}{5cm}
 \vspace{-5pt}
  \begin{tikzpicture}
  \draw (-2,0) -- (0,0);
  \draw (-4.5,0.5) -- (-2.5,0.5);
  \draw (-4.5,-0.5) -- (-2.5,-0.5);
  \draw[fill=white] (-2,0) circle (0.2ex);
  \draw[fill=black] (-2.5,0) circle (0.2ex);
  \draw[fill=black] (-2.5,0.5) circle (0.2ex);
  \draw[fill=black] (-2.5,-0.5) circle (0.2ex);
  \draw (-3.5,0.4) edge[ <->, bend right=10] node[auto] {$\tau$} (-3.5,-0.4) ;  
  \end{tikzpicture}
  \caption{$\tau$ acting on $\Y$}
  \vspace{-20pt}
 \end{wrapfigure}
Next, consider the reflection:
 \begin{equation*}
  \begin{array}{cccc}
  \tau:&\R^3&\to&\R^3\\
  &(x,y,z)&\mapsto&(x,-y,-z),
  \end{array}
 \end{equation*}
and note that $\tau$ is a Poisson involution, i.e.\ \[\tau_*(\pi)=\pi,\ \ \tau^2=\mathrm{id},\] 
and it interchanges the leaves $S^{+}_{\lambda}$ and $S^{-}_{\lambda}$, $\lambda \leq 0$. We denote also by $\tau$ the diffeomorphism induced on $\Yr$.
Note that $\tau$ normalizes $G_N^0$ and $\mathrm{Diff}_0^0(\Yr)$, and in fact:
\[\tau\cdot \widetilde{\phi}\cdot  \tau=\widetilde{\tau \cdot \phi\cdot  \tau}, \ \ \tau \cdot (\phi_1,\phi_2)\cdot \tau= (\phi_2,\phi_1).\]
Therefore the following groups are isomorphic: 
\begin{align*}
G_N&:=G_N^0\, \cup\, G_N^0\cdot \tau, \\
\mathrm{Diff}_0(\Yr)&:=
\mathrm{Diff}_0^0(\Yr)\,
\cup\, \mathrm{Diff}_0^0(\Yr)\cdot \tau.
\end{align*}

Note that $G_N$ normalizes $G_T$:
\begin{align*}
\exp(\eta T)\cdot \widetilde{\phi}&=\widetilde{\phi}\cdot \exp(\widetilde{\phi}^*(\eta)\, T),\\
\exp(\eta T)\cdot \tau& =\tau\cdot  \exp(-\eta T),
\end{align*}
where $\widetilde{\phi}^*(\eta)=
\widetilde{\phi^*h}$, for $\eta=\widetilde{h}$.

We obtain the group $G_T\rtimes G_N$, which in principle is the group of outer-automorphisms of the Poisson manifold. It would be interesting to know whether this is a correct interpretation: 
\begin{question}
Is the natural map $G_T\rtimes G_N\to \mathrm{Poiss}/\mathrm{Ham}$ an isomorphism? Equivalently, is it true that:
\[\mathrm{Poiss}=\mathrm{Ham}\rtimes G_T\rtimes G_N?\]
\end{question}

\subsection{Deformations}

The second Poisson cohomology has the heuristic interpretation of being the ``tangent space" to the Poisson-moduli space.
In our case, by Theorem \ref{main theorem}, every class in $H^2(\mathfrak{sl}_2^*(\mathbb{R}),\pi)$ has a unique representative of the form 
\[\eta N\wedge T,\ \textrm{with} \ \ \eta\in C^{out}.\]
Since the Schouten bracket is trivial on these elements, it follows that
\[\pi_{\eta}:=\pi+\eta N\wedge T, \ \ \textrm{with}\ \ \eta\in C^{out}\]
is a Poisson structure. In other words, infinitesimal deformations are unobstructed. Note that these are precisely the deformations of $\pi$ constructed by Weinstein in \cite[Prop 6.3]{Wein83} to show that $\spl$ is smoothly degenerate. The Poisson structure $\pi_{\eta}$ differs from $\pi$ only on $O=\{f> 0\}$. Using the coordinates $(\theta,w,f)$ from \eqref{coordinates O} on $O$ and writing $\eta=h\circ f$, with $\supp h\subset [0,\infty)$, we have that:
\[\pi_{\eta}|_O=\big(\partial_{\theta}+h\circ f\,\partial_{f} \big)\wedge\partial_w.\]
Note that the leaves of $\pi_{\eta}$ are
perturbations of the cylinders $S_{\lambda}$. In order to understand their shape, note that the leaves of $\pi_{\eta}$ cut the plane $z=0$ in the flow lines of the Hamiltonian vector field of $-z$:
\begin{align*}
-\pi_{\eta}^{\sharp}(\dif z)|_{z=0}&=\left(\partial_{\theta}+h\circ f\,\partial_{f}\right)\big|_{z=0}=\partial_{\theta}+\frac{h(r^2)}{2r^2}\, r\partial_{r}\\
&=\left(x\partial_y-y\partial_x \right)+u(x^2+y^2)\left(x\partial_x+y\partial_y \right),
\end{align*}
where $u(t):=h(t)/(2t)$ is smooth and vanishes flatly at $t=0$. The shape of the flow lines depends on the behaviour of $u$; for example, if $u(r^2)=0$, then the circle of radius $r$ is an orbit; if $u(t)<0$ for $t\in (0,r^2)$, then the flow lines in the disk of radius $r$ spiral towards the origin. 

It would be interesting to know if all deformations are of this type:
\begin{question}
Is every Poisson structure near $\pi$ isomorphic to $\pi_{\eta}$ for some $\eta\in C^{out}$?   
\end{question}
There are options in how to formulate this question precisely: for example, one can consider deformations on a small ball around $0$, or one can consider global Poisson structures which are close with respect to the Whitney (open-open) $C^{\infty}$-topology. A related problem is: 
\begin{question}
Is every Poisson structure with isotropy Lie algebra $\spl$ at a zero isomorphic to $\pi_{\eta}$, for some $\eta\in C^{out}$?   
\end{question}

Further, we note that different functions $\eta\in C^{out}$ can yield isomorphic Poisson structures $\pi_{\eta}$. Infinitesimally, this phenomenon arises because the Schouten-Nijenhuis bracket bracket in cohomology is non-trivial (see \eqref{Schouten}) in degrees $(1,2)\mapsto 2$, and this operation encodes the infinitesimal action of outer-automorphisms on deformations. In fact, only elements $\chi N\in \mathfrak{g}_{N}$, with $\chi\in C^{out}$ act non-trivially. Via the isomorphism \eqref{g_N}, this subalgebra corresponds to the following subalgebra of vector fields on $\Yr$:
\[\mathfrak{X}^1_{\to}(\R)=\big\{h\partial_t\ |\ h|_{(-\infty,0]}=0\big\}\subset \mathfrak{X}_0^1(\Yr),\]
with corresponding subgroup: 
\[\mathrm{Diff}_{\to}(\R)=\big\{\phi\in \mathrm{Diff}(\R)\ | \ \phi|_{(-\infty,0]}=\mathrm{id}\big\}\subset \mathrm{Diff}_{0}(\Yr),\]
where in both cases we use the diagonal inclusion. The action of $\widetilde{\phi}$, with $\phi\in \mathrm{Diff}_{\to}(\R)$, on $\pi_{\eta}$, with $\eta=h\circ f$, is given by: 
\[\widetilde{\phi}^*(\pi_{\eta})=  \pi_{\eta'}, \ \ \ \eta'=\frac{h\circ \phi}{\phi'}\circ f \in C^{out}.\]
Note also that $\tau$ acts non-trivially:
\[\tau^*(\pi_{\eta})=\pi_{-\eta}.\] 
It would be interesting to know whether these are all identifications: 
\begin{question}
For $i=1,2$, consider $h_i\in C^{\infty}(\R)$, with $\supp(h_i)\subset [0,\infty)$, and let $\eta_i:=h_i\circ f$. If the Poisson structures
$\pi_{\eta_1}$ and $\pi_{\eta_2}$ are isomorphic, does there exist $\phi\in \mathrm{Diff}_{\to}(\R)$ such that 
\[h_2=\frac{h_1\circ \phi}{\phi'}\ \ \ \textrm{or}\ \ \ h_2=-\frac{h_1\circ \phi}{\phi'}\ ?\]
\end{question}
These formulas come from the adjoint action of $\mathrm{Diff}_{\to}(\R)$. Namely, if $\phi\in \mathrm{Diff}_{\to}(\R)$ and $h\partial_t\in \mathfrak{X}_{\to}^1(\R)$, then 
\[\phi^*(h\partial_t)=\frac{h\circ \phi}{\phi'}\partial_t.\]
Thus, we obtain a bijection:
\[H^2(\mathfrak{sl}_2^*(\mathbb{R}),\pi)/G_N\simeq \big(\mathfrak{X}_{\to}^1(\R)/\mathrm{Diff}_{\to}(\R)\big)/\{\pm 1\},\]
where the right hand-side can be thought of as adjoint orbits of 
$\mathrm{Diff}_{\to}(\R)$, up to $\pm 1$. Assuming that the answers to the last two questions are positive, this space is a model for the Poisson-moduli space around $\pi$.

\subsection{The Koszul-Brylinski double complex}
Dual to the Poisson complex of a Poisson manifold $(M,\pi)$, 
Koszul \cite{Kos85} introduced a differential on differential forms:
\begin{align*}
    \delta_{\pi}:=\iota_{\pi}\circ \dif -\dif \circ \iota_{\pi}: \Omega^{\bullet}(M)\to \Omega^{\bullet-1}(M), \ \ \ \ \delta_{\pi}^2 =0,
\end{align*}
which yields the Poisson homology groups: $H_{\bullet}(M,\pi)$. Moreover, one has
\begin{align*}
    \dif\circ \delta_{\pi} +\delta_{\pi}\circ \dif =0,
\end{align*}
and therefore we have a bidifferential complex $(\Omega^{\bullet}(M),\dif,\delta_{\pi})$. In \cite{Bry88}, Brylinski gave a more explicit formula of $\delta_{\pi}$ and studied this complex in more detail. Moreover, by \cite{Xu99} and \cite{ELuWein}, for an oriented, unimodular Poisson manifold, the contraction with an $\mathfrak{ham}$-invariant volume form $\mu$ gives an isomorphism
between the Poisson cohomology complex and the Poisson homology complex:
\begin{align*}
\mu^{\flat}: (\mathfrak{X}^{k}(M),\dif_{\pi})\diffto (\Omega^{m-k}(M),\delta_{\pi}).
\end{align*}

For $(\mathfrak{sl}_2^*(\mathbb{R}) ,\pi)$ the standard volume form $\mu=\dif x\wedge \dif y \wedge \dif z$ is invariant. Applying contraction with $\mu$ on the representatives for Poisson cohomology from Theorem \ref{main theorem} we obtain:
\begin{align*}
H_3(\mathfrak{sl}_2^*(\mathbb{R}),\pi)\ \ &\simeq\ \  \Cas(\mathfrak{sl}_2^*(\mathbb{R}))\cdot \mu,\\
H_2(\mathfrak{sl}_2^*(\mathbb{R}),\pi)\ \ &\simeq\ \  C^{out}\cdot\dif f\wedge \dif \theta\ \oplus\  C^{flat}\cdot \overline{\omega},\\
H_1(\mathfrak{sl}_2^*(\mathbb{R}),\pi)\  \ &\simeq\  \  C^{out}\cdot \dif \theta,\\
H_0(\mathfrak{sl}_2^*(\mathbb{R}),\pi)\ \ &\simeq\  \ \R[[f]],
\end{align*}
where $\overline{\omega}:=\mu^{\flat}(N)$ is a closed extension to $\R^3\backslash Z$ of the leaf-wise symplectic form. Using this, we calculate the de Rham cohomology of the Poisson homology:
\begin{corollary}
We have that:
\[
H_{DR}^{k}(H_{\bullet}(\mathfrak{sl}_2^*(\mathbb{R}),\pi),\dif)\simeq
\begin{cases}
\R[[f]], & k=0, \ \textrm{or}\ \ k=3;\\
0, & k=1, \ \textrm{or}\ \ k=2.\\
\end{cases}\]
\end{corollary}



\section{Formal Poisson cohomology of $\mathfrak{sl}_2^*(\mathbb{R})$}\label{section: formal}

We begin this section by introducing flat and formal Poisson cohomology. For the linear Poisson structure on the dual of a Lie algebra, we identify these cohomologies with the Chevalley-Eilenberg cohomology of the Lie algebra with coefficients in certain representations. Then we specialize to semi-simple Lie algebras, for which, using standard results from Lie theory, we calculate the formal Poisson cohomology (Proposition \ref{generators}); and explicitly, for $\spl$. An important consequence (Corollary \ref{ses smooth formal}) is that, for semi-simple Lie algebras, the calculation of Poisson cohomology can be reduced to that of flat Poisson cohomology. 

\subsection{Flat and formal Poisson cohomology}

Let $(M,\pi)$ be a Poisson manifold. Its Poisson cohomology $H^{\bullet}(M,\pi)$ is the cohomology of the chain complex: 
\[(\mathfrak{X}^{\bullet}(M),\dif_{\pi}:=[\pi,\cdot]).\]
For $p\in M$, let $\mathfrak{X}_p^{\bullet}(M)$ denote the set of multivector fields that are flat at $p$. Since $\mathfrak{X}_p^{\bullet}(M)$ is a Lie ideal in $\mathfrak{X}^{\bullet}(M)$, it is also a subcomplex with respect to $\dif_{\pi}$. The cohomology of this complex, denoted $H^{\bullet}_{p}(M,\pi)$, will be called the flat Poisson cohomology at $p$. 

By Borel's Lemma on the existence of smooth functions with a prescribed Taylor series, we have the following identification for the quotient:
\[\mathfrak{X}^{\bullet}(M)/\mathfrak{X}_p^{\bullet}(M)\simeq \wedge^{\bullet}T_pM\otimes \R[[T^*_pM]],\]
where $\R[[T^*_pM]]$ denotes the algebra of formal power series of functions at $p$. Thus, we obtain a short exact sequence of complexes
 \begin{align*}
  0\to(\mathfrak{X}^{\bullet}_p(M),\dif_{\pi})\to(\mathfrak{X}^{\bullet}(M),\dif_{\pi})\stackrel{j^{\infty}_p}{\longrightarrow} (\wedge^{\bullet}T_pM\otimes \R[[T^*_pM]],\dif_{j^{\infty}_{p}\pi})\to 0,
 \end{align*}
where $j^{\infty}_p$ is the infinite jet map. The cohomology of the quotient complex, denoted by $H^{\bullet}_{F,p}(M,\pi)$, will be called the formal Poisson cohomology at $p$. The short exact sequence induces a long exact sequence in cohomology: 
\begin{equation}\label{jet}
\ldots \stackrel{j^{\infty}_p}{\to} H^{q-1}_{F,p}(M,\pi)\stackrel{\partial}{\to} 
H^{q}_{p}(M,\pi)\to H^{q}(M,\pi)\stackrel{j^{\infty}_p}{\to}H^{q}_{F,p}(M,\pi )\stackrel{\partial}{\to}\ldots.
 \end{equation}

\subsection{Poisson cohomology of linear Poisson structures}

A Poisson structure on a vector space is called linear if the set of linear functions is closed under the Poisson bracket. Such Poisson structures are in one-to-one correspondence with Lie algebra structures on the dual vector space. Namely, let $(\mathfrak{g},[\cdot,\cdot])$ be a real, finite-dimensional Lie algebra. The associated linear Poisson structure $\pi$ on $\mathfrak{g}^*$ is determined by the condition that the map $l:\mathfrak{g}\to C^{\infty}(\mathfrak{g}^*)$, which identifies $\mathfrak{g}$ with $(\mathfrak{g}^*)^*$, 
is a Lie algebra homomorphism:
\[\{l_X,l_Y\}=l_{[X,Y]},\]
where $\{\cdot,\cdot\}$ is the Poisson bracket on $C^{\infty}(\mathfrak{g}^*)$ corresponding to $\pi$. In particular, $C^{\infty}(\mathfrak{g}^*)$ becomes a $\mathfrak{g}$-representation, with $X\cdot f:=\{l_X,f\}$. Moreover, the Poisson complex of $(\mathfrak{g}^*,\pi)$ is isomorphic to the Chevalley-Eilenberg complex of $\mathfrak{g}$ with coefficients in $C^{\infty}(\mathfrak{g}^*)$ \cite[Prop 7.14]{Laurent2013}
\begin{equation}\label{iso_complexes}
(\mathfrak{X}^{\bullet}(\mathfrak{g}^*),\dif_{\pi})\simeq (\wedge^{\bullet}\mathfrak{g}^*\otimes C^{\infty}(\mathfrak{g}^*),\dif_{EC}).
\end{equation}
This identification allows for the use of techniques from Lie theory in the calculation of Poisson cohomology, as we will do in the sequel. 

Suppose that $R$ is a representation of $\g$, and denote by $R^{\g}\subset R$ the set of $\g$-invariant elements. Since $R^{\g}$ is a trivial subrepresentation, we can identify \[H^q(\mathfrak{g})\otimes R^{\mathfrak{g}} \simeq  H^q(\mathfrak{g},R^{\mathfrak{g}}).\] 
Moreover, the inclusion $\iota:R^{\g}\hookrightarrow R$ induces a map in cohomology:
 \begin{align}\label{subrep}
  \iota_{*} : H^q(\mathfrak{g})\otimes R^{\mathfrak{g}} \to H^q(\mathfrak{g},R).
 \end{align}
For $q=0$ this map is always an isomorphism: $R^{\g}\simeq H^0(\g,R)$.
In general, $\iota_{*}$ need not be injective nor surjective. However, by \cite[Thm 13]{Hoch1953}, if $\g$ is semisimple and $R$ is finite-dimensional, then \eqref{subrep} is an isomorphism for all $q\geq 0$. The same conclusion holds also in the following more general situation, which we will use in the next subsection:
\begin{lemma}
\label{iso cohomology}
Let $\mathfrak{g}$ be a semisimple Lie algebra. Let $R=\prod_{\alpha\in  A}R_{\alpha}$ be a direct product of finite-dimensional $\g$-representations. Then the map in \eqref{subrep} is an isomorphism for all $q\geq 0$.
\end{lemma}
\begin{proof}
Since $\g$ is finite-dimensional, the Chevalley-Eilenberg complex of $R$ is canonically isomorphic to the direct product of complexes: 
\begin{equation}\label{eq:prod_complexes}
(\wedge^{\bullet}\g\otimes R, \dif_{EC})\simeq \prod_{\alpha\in A}(\wedge^{\bullet}\g\otimes R_{\alpha}, \dif_{EC}),    
\end{equation}
which yields an isomorphism in cohomology $H^{\bullet}(\g, R)\simeq \prod_{\alpha\in A}H^{\bullet}(\g, R_{\alpha})$. Moreover, under the isomorphism \eqref{eq:prod_complexes}, the subcomplexes of invariant elements are in one-to-one correspondence: 
\[(\wedge^{\bullet}\g\otimes R^{\g}, \dif_{EC})\simeq \prod_{\alpha\in A}(\wedge^{\bullet}\g\otimes R^{\g}_{\alpha}, \dif_{EC}),\]
and therefore $H^{\bullet}(g)\otimes R^{g}\simeq \prod_{\alpha\in A}H^{\bullet}(\g)\otimes R^{\g}_{\alpha}$. This identifies the map $\iota_*$ for $R$ with the direct product of the maps $\iota_{\alpha *}$ for $R_{\alpha}$:
\[\iota_*\simeq \prod_{\alpha\in A}\iota_{\alpha *}: \prod_{\alpha\in A}H^{\bullet}(\g)\otimes R^{\g}_{\alpha}\to \prod_{\alpha\in A}H^{\bullet}(\g, R_{\alpha}).\]
Since $\g$ is semisimple and all $R_{\alpha}$'s are finite-dimensional, 
each $\iota_{\alpha *}$ is an isomorphism \cite[Thm 13]{Hoch1953}, and therefore so is their product $\iota_*$.
 \end{proof}

We are interested in the representation $R=C^{\infty}(\g^*)$, whose space of invariants are the Casimir functions:
\begin{align*}
H^0(\g^*,\pi)=\Cas(\g^*)=C^{\infty}(\g^*)^{\g}.
\end{align*}
If $\g$ is semisimple, then the map \eqref{subrep} for $R=C^{\infty}(\g^*)$ is an isomorphism for all $q\geq 0$ if and only if the Lie algebra is compact. Namely, by the construction in \cite{Wein87}, 
for all non-compact semisimple Lie algebra it fails at $q=1$. For compact Lie algebras, this was proven in \cite[Thm 3.2]{GW92}, and so, by \eqref{iso_complexes}, we have that:
\[H^{q}(\g^*,\pi)=H^{q}(\g,C^{\infty}(\g^*))\simeq  H^q(\g)\otimes\Cas(\g^*).\]

\subsection{Formal Poisson cohomology of linear Poisson structures}

Under the isomorphism \eqref{iso_complexes}, the subcomplex of multivector fields on $\g^*$ that are flat at $0$ corresponds to the Eilenberg-Chevalley complex of $\g$ with coefficients in the subrepresentation $C_0^{\infty}(\g^*)\subset C^{\infty}(\g^*)$ consisting of smooth functions that are flat at zero:
\[\mathfrak{X}^{\bullet}_0(\mathfrak{g}^*)\simeq \wedge^{\bullet}\mathfrak{g}^*\otimes C^{\infty}_0(\mathfrak{g}^*).\]
Therefore, the quotient complex is naturally identified with 
\[\wedge^{\bullet} \g^*\otimes \R[[\g]],\]
where $\R[[\g]]=C^{\infty}(\mathfrak{g}^*)/C^{\infty}_0(\mathfrak{g}^*)$ is the ring of formal power series of functions on $\g^*$. Thus, the formal Poisson cohomology at $0\in \g^*$, which for simplicity we denote by $H^{\bullet}_{F}(\g^* ,\pi )$, is naturally isomorphic to the cohomology of $\g$ with coefficients in the representation $\R[[\g]]$
\[H^{\bullet}_{F}(\g^*)\simeq H^{\bullet}(\g,\R[[\g]]).\]
As a representation, $\R[[\g]]$ is isomorphic to the product of the symmetric powers of the adjoint representation: 
  \[\R[[\g]]\simeq \prod_{k\geq 0} S^k(\g).\]
Thus Lemma \ref{iso cohomology} implies: 
\begin{proposition}\label{formal cohomology}
For the formal Poisson cohomology at $0$ of a semisimple Lie algebra $\g$ we have that: 
 \begin{align*}
  H^{\bullet}_{F}(\g^*,\pi)&\simeq H^{\bullet}(\g)\otimes \Cas _{F}(\g^*),
 \end{align*}
 where $\Cas_{F}(\g^*)=\R[[\g]]^{\g}$ is the 
 set of formal Casimir functions.
 \end{proposition}

On the other hand, the space of formal Casimir functions is well-understood: 

\begin{proposition}\label{generators}
Let $(\g^* ,\pi ) $ be the linear Poisson structure associated to a semisimple Lie algebra $\g$. Then there exist $n=\dim \g - \max \mathrm{rank}(\pi)$ algebraically independent homogeneous polynomials $f_1 ,\dots , f_n$ such that
\begin{align*}
\Cas_F(\g^*)=\R[[f_1 ,\dots ,f_n ]] \subset \R[[\g]],
\end{align*}
where $\R[[f_1,\dots ,f_n]]$ denotes formal power series in the polynomials $f_i$.
\end{proposition}
\begin{proof}
 \cite[Thm 7.3.8]{Dixmier96} gives such polynomials $f_1 ,\dots ,f_n$ which generate the algebra of $\g$-invariant polynomials $\R[\g]^{\g }$. Clearly $\R[[f_1 ,\dots ,f_n ]]\subset \Cas_F(\g^*)$. For the other inclusion, let $g\in \Cas_{F}(\g^*)$. Let $g_k\in S^k(\g)$ denote the homogeneous component of degree $k$ of $g$. Since $g_k\in S(\g)^{\g}$ and the $f_i$'s are algebraically independent, there is a unique polynomial $p_k\in \R[x_1,\ldots,x_n]$ such that $g_k=p_k(f_1,\ldots,f_n)$. Note that each monomial of $p_k$ has total degree at least $k/D$, where $D:=\mathrm{max}\ \mathrm{degree}(f_i)$. Therefore, $p=\sum_{k\geq 0}p_k$ represents an element in $\R[[x_1,\ldots,x_n]]$, which satisfies $g=p(f_1,\ldots,f_n)$.
  \end{proof}
  
The invariant polynomials on $\mathfrak{sl}_2^*(\mathbb{R})$ are generated by the function $f$ from \eqref{Casimir}. We conclude:

\begin{corollary}\label{formal cohomology sl2} 
The formal Poisson cohomology of $\mathfrak{sl}_2^*(\mathbb{R})$ at 0 is given by
\[H^0_F(\mathfrak{sl}_2^*(\mathbb{R}),\pi)=\R[[f]], \ \ H^1_F(\mathfrak{sl}_2^*(\mathbb{R}),\pi)=0,\]
\[H^2_F(\mathfrak{sl}_2^*(\mathbb{R}),\pi)=0, \ \ H^3_F(\mathfrak{sl}_2^*(\mathbb{R}),\pi)=\R[[f]] \otimes \partial_x\wedge\partial_y\wedge \partial_z.\]
 \end{corollary}
\begin{proof}
By the previous propositions $H^{\bullet}_F(\mathfrak{sl}_2^*(\mathbb{R}),\pi)= H^{\bullet}(\spl)\otimes \R[[f]]$. Clearly $H^{0}(\spl)=\R$, and it is easy to see that $H^{3}(\spl)=\R \partial_x\wedge\partial_y\wedge \partial_z$. In degrees 1 and 2, the conclusion follows by the Whitehead lemma, which states that for a semisimple Lie algebra $\g$, $H^1(\g)=0$ and $H^2(\g)=0$. 
\end{proof}

The previous propositions reduce the calculation of Poisson cohomology of a semisimple Lie algebra to that of flat Poisson cohomology at 0:
\begin{corollary}\label{ses smooth formal}
For a semi-simple Lie algebra $\g$, the Poisson cohomology of $(\g^*,\pi)$ fits into the short exact sequence 
\[0\to H^{\bullet}_0(\g^*,\pi)\to H^{\bullet}(\g^*,\pi)\stackrel{j^{\infty}_0}{\to} H^{\bullet}_F(\g^*,\pi)\to 0.\]
\end{corollary}
\begin{proof}
Using the long exact sequence \eqref{jet}, it suffices to show that $j^{\infty}_0$ is surjective in cohomology. By Proposition \ref{formal cohomology}, every element in $H^{\bullet}_F(\g^*,\pi)$ has a representative of the form $w=\sum_i \omega_i\otimes g_i$, where $\omega_i\in \wedge^{\bullet}\g$ are closed elements, and $g_i\in \Cas_F(\g^*)$. By Proposition \ref{generators} we can write $g_i=h_i(f_1,\dots ,f_n)$, for some $h_i\in \R[[x_1,\dots ,x_n ]]$. By Borel's lemma, there are smooth functions $\overline{h}_i\in C^{\infty}(\R^n)$ such that $j^{\infty}_0 \overline{h}_i=h_i$. Therefore \[\overline{w}=\sum_i \omega_i\otimes \overline{h}_i(f_1,\dots ,f_n)\in \wedge^{\bullet}\g\otimes C^{\infty}(\g^*)\simeq \mathfrak{X}^{\bullet}(\g^*)\]
is a closed element satisfying $j^{\infty}_0 \overline{w}=w$.
\end{proof}
\begin{remark}
The versions of Propositions \ref{formal cohomology} and \ref{generators} with polynomials instead of formal power series are also valid (see \cite{Laurent2013} and \cite[Thm 7.3.8]{Dixmier96}). We expect these results to hold also for the algebra of analytic functions $C^{\omega}(\g^*)$, with a possible proof using the techniques from \cite{Conn84}.
 \end{remark}
 
\section{Flat Poisson cohomology of $\mathfrak{sl}_2^*(\mathbb{R})$}\label{section: flat PC}

The Poisson manifold $\mathfrak{sl}_2^*(\mathbb{R})\backslash\{0\}$ is regular, of corank one, and unimodular. Such Poisson structures can be described in terms of foliated cohomology \cite{Vaisman,Gammella}. We will explain this in the first two subsections. In the third subsection, we introduce the flat foliated cohomology of $\mathfrak{sl}_2^*(\mathbb{R})$, which we use in the last subsection to calculate the flat Poisson cohomology. This, together with Corollaries \ref{formal cohomology sl2} and \ref{ses smooth formal}, complete the description of the Poisson cohomology $\mathfrak{sl}_2^*(\mathbb{R})$ from Theorem \ref{main theorem}. The calculation of the flat foliated cohomology will be left for Sections \ref{section flat foli} and \ref{section: analysis}, and is the most technically involved part of the paper.

\subsection{Foliated cohomology}

We will follow \cite[Chp 1]{Torres15}. 
For a regular foliation $\mathcal{F}\subset TM$ on a manifold $M$, we denote the complex of foliated forms by 
\[(\Omega^{\bullet}(\mathcal{F}), \dif_{\mathcal{F}});\]
i.e.\ $\Omega^{\bullet}(\mathcal{F}):=\Gamma(\wedge^{\bullet} \mathcal{F}^*)$ consists of smooth families of differential forms on the leaves of $\mathcal{F}$, and $\dif_{\mathcal{F}}$ is the leafwise de Rham differential. The resulting cohomology is called the foliated cohomology of $\F$, and is denoted by $H^{\bullet}(\F)$. The normal bundle to $\F$, denoted by $\nu:=TM/\F$, carries the Bott-connection
\[   \nabla:\Gamma(\F)\times \Gamma(\nu)\to \Gamma(\nu),\ \ \ \  \nabla_X(\overline{V}):=\overline{[X,V]},
 \]
which, via the usual Koszul-type formula, induces a differential on $\nu$-valued forms $(\Omega^{\bullet}(\F,\nu),\dif_{\nabla})$, which yields the cohomology of $\F$ with values in $\nu$, denoted $H^{\bullet}(\F,\nu)$. Similarly, the dual connection on $\nu^*$ gives rise to the complex $(\Omega^{\bullet}(\F,\nu^*),\dif_{\nabla})$, with cohomology groups $H^{\bullet}(\F,\nu^*)$. 

Assume now that $\F$ has codimension one. Then we have the short exact sequence of complexes: 
\begin{equation}\label{ses fol-deRham}
0\to (\Omega^{\bullet-1}(\F,\nu^*),\dif_{\nabla})\to (\Omega^{\bullet}(M),\dif)\stackrel{r}{\to}(\Omega^{\bullet}(\F),\dif_{\F})\to 0,    
\end{equation}
where $r$ is the restriction map, and an element $\eta$ in the kernel of $r$ is canonically identified with the element $u(\eta)\in \Omega^{\bullet-1}(\F,\nu^*)$, defined by: 
\[\langle u(\eta),\overline{V}\rangle= r(i_{V}\eta), \ \ \overline{V}\in \nu.\]
Assume in addition that $\nu^*$ is orientable, and let $\varphi\in \Gamma(\nu^*)\subset \Omega^1(M)$ be a defining 1-form for $\F$, i.e.\ $\varphi$ is nowhere zero and $\F=\ker (\varphi)$. Then $\ker(r)$ is the differential ideal generated by $\varphi$
\begin{equation}\label{the complex}
(\varphi\wedge\Omega^{\bullet-1}(M),\dif)
\end{equation}
The foliation is called unimodular, if there exists a defining one-form $\varphi$ which is closed: $\dif \varphi=0$. Such a one-form is parallel for the dual of the Bott-connection, and it induces an isomorphism of complexes: 
\[(\Omega^{\bullet}(\F),\dif_{\F})\diffto (\Omega^{\bullet}(\F,\nu^*),\dif_{\nabla}), \ \ \eta\mapsto \varphi \otimes  \eta.\]
Similarly, the dual of $\varphi$ gives a parallel section of $\nu$, and we obtain an isomorphism of complexes:
\[(\Omega^{\bullet}(\F,\nu),\dif_{\nabla})\diffto (\Omega^{\bullet}(\F),\dif_{\F}), \ \ \eta\mapsto \langle \varphi, \eta\rangle.\]

\subsection{Cohomology of codimension one symplectic foliations}

Let $(M,\pi)$ be a regular Poisson manifold of corank one, and denote its symplectic foliation by $(\mathcal{F},\omega)$, where $\F=\pi^{\sharp}(T^*M)$ and $\omega\in \Omega^2(\F)$ is the leafwise symplectic structure. The Poisson complex of $\pi$ fits into a short exact sequence:  
 \begin{align}\label{ses poisson}
  0\to (\Omega^{\bullet}(\mathcal{F}),\dif_{\mathcal{F}})\xrightarrow{j} (\mathfrak{X}^{\bullet}(M),\dif_{\pi})\xrightarrow{p} (\Omega^{\bullet-1}(\mathcal{F},\nu),\dif_{\nabla})\to 0.
 \end{align}
Regarding the Poisson complex as the de Rham complex of the Lie algebroid $T^*M$, the map $j$ is obtained by pulling back Lie algebroid forms via the Lie algebroid map $\pi^{\sharp}:T^*M\to \F$; explicitly, 
\[j(\eta)=(-\pi^{\sharp})(\eta),\]
where we denoted by 
\[(-\pi^{\sharp}):\wedge^{\bullet} \F^*\diffto \wedge^{\bullet}\F\]
the isomorphism induced by $-\pi^{\sharp}$. For the cokernel, we have the canonical isomorphism $\wedge^{\bullet} TM/\wedge^{\bullet} \F\simeq  \nu\otimes \wedge^{\bullet-1}\F$; and the map $p$ is obtained by using the isomorphism
\[(-\omega_b)=(-\pi^{\sharp})^{-1}:\wedge^{\bullet} \F \diffto \wedge^{\bullet} \F^*.\] 
Explicitly, 
\[ \langle \alpha, p(w)\rangle:=(-\omega_b) (i_{\alpha}w), \ \ \alpha\in \nu^*,\]
where we note that, $i_{\alpha}w\in \wedge^{k-1}\F$. Therefore there is a long exact sequence 
 \begin{equation}\label{les poisson}\arraycolsep=1.4pt
  \begin{array}{ccccccccccc}
   \dots &\xrightarrow{\partial}& H^k(\mathcal{F}) &\xrightarrow{j}&H^k(M,\pi)&\xrightarrow{p}&H^{k-1}(\mathcal{F},\nu)&\xrightarrow{\partial}&H^{k+1}(\mathcal{F})&\to&\dots
  \end{array}
 \end{equation}
The boundary morphism $\partial$ is up to a sign given by the cup-product with the class $\dif_{\nu}[\omega]\in H^2(\mathcal{F},\nu^*)$, where $\dif_{\nu}:H^{\bullet}(\F)\to H^{\bullet}(\F,\nu^*)$ is the boundary map of the long exact sequence associated to \eqref{ses fol-deRham}. This class has the geometric interpretation of being the transverse variation of the leafwise symplectic form. 
 
Assume now that $\F$ is coorientable and unimodular, and let $\varphi$ be a closed defining one-form. Then $\varphi$ gives flat trivializations of the bundles $\nu$ and $\nu^*$, and the cohomology of $\F$ with trivial coefficients and with coefficients in these bundles can all be calculated using the subcomplex \eqref{the complex} of the de Rham complex, which will be useful in our situation. Thus \eqref{ses poisson} can be rewritten as the short exact sequence:
 \begin{align}\label{ses poisson unimod}
  0\to (\varphi\wedge \Omega^{\bullet}(M),\dif)\xrightarrow{j_{\varphi}} (\mathfrak{X}^{\bullet}(M),\dif_{\pi})\xrightarrow{p_{\varphi}} (\varphi\wedge \Omega^{\bullet-1}(M),\dif)\to 0.
 \end{align}
Unravelling the identifications made above, the maps $j_{\varphi}$ and $p_{\varphi}$ can be made explicitly. Namely, let $V$ be a vector field on $M$ such that $i_V\varphi=1$, and let $\widetilde{\omega}\in \Omega^2(M)$ be the unique extension of $\omega$ such that $i_{V}\widetilde{\omega}=0$. 
Then 
\[j_{\varphi}=(-\pi^{\sharp})\circ i_{V}, \ \ \ \  p_{\varphi}=e_{\varphi}\circ (-\widetilde{\omega}_b)\circ i_{\varphi},\]
where $e_{\varphi}(-)=\varphi\wedge (-)$ is the exterior product with $\varphi$. Even though it was convenient to use $V$ (and $\widetilde{\omega}$) to write these formulas, the maps $j_{\varphi}$ and $p_{\varphi}$ are independent of this choice. However, $V$ allows us to build dual maps:
\[  0\leftarrow \varphi\wedge \Omega^{\bullet}(M)\xleftarrow{p_{V}} \mathfrak{X}^{\bullet}(M)\xleftarrow{j_{V}} \varphi\wedge \Omega^{\bullet-1}(M)\leftarrow 0,\]
with
\begin{equation}\label{dual_maps}
p_V=e_{\varphi}\circ (-\widetilde{\omega}_b), \ \ \ j_V=e_V\circ (-\pi^{\sharp})\circ i_V,
\end{equation}
which satisfy the homotopy relations:
\begin{equation}\label{dual_maps_relations}
p_V\circ j_V=0,\ \ p_V\circ j_{\varphi}=\mathrm{Id}, \ \ p_{\varphi}\circ j_{V}=\mathrm{Id}, \ \ \mathrm{Id}=j_V\circ p_{\varphi}+j_{\varphi}\circ p_V.
\end{equation}
It can be checked that the maps $p_V$ and $j_V$ are chain morphisms precisely when $V$ is a Poisson vector field, which is also equivalent $\widetilde{\omega}$ being closed; in this case the pair $(\varphi,\widetilde{\omega})$ is a cosymplectic structure on $M$. In general, we can write $\dif\widetilde{\omega}=\varphi\wedge \xi$, where $\xi=i_V\dif\widetilde{\omega}$. Even though we will not use this later, let us remark that the boundary map for the long exact sequence in cohomology induced by \eqref{ses poisson unimod} is given up to sign by the chain map:
\[  e_{\xi}:(\varphi\wedge \Omega^{\bullet-1}(M),\dif)\to (\varphi\wedge \Omega^{\bullet+1}(M),\dif), \ \ \eta\mapsto \xi\wedge \eta.\]
 
\subsection{A short exact sequence for the flat Poisson complex}
 
If we remove the origin from the Poisson manifold $(\mathfrak{sl}_2^*(\mathbb{R}),\pi)$, we obtain a codimension one symplectic foliation $(\F,\omega)$ which is unimodular with defining one-form $\dif f$. Therefore, the techniques from the previous section can be used to describe its cohomology.
Consider the vector field on $\R^3\backslash\{0\}$
 \begin{align*}
   V:=&\frac{1}{2(x^2+y^2+z^2)}(x\partial_x+y\partial_y-z\partial_z),
  \end{align*}
and note that $\dif f (V)=1$. The unique extension $\widetilde{\omega}$ of the leafwise symplectic structure satisfying $i_V\widetilde{\omega}=0$ is given by 
  \begin{align*}
   \widetilde{\omega}:=-\frac{1}{x^2+y^2+z^2}(x\dif y\wedge \dif z +y\dif z\wedge \dif x -z\dif x \wedge \dif y).
  \end{align*}
Therefore, the Poisson complex of $(\mathfrak{sl}_2^*(\mathbb{R})\backslash\{0\},\pi)$ fits into the short exact sequence \eqref{ses poisson unimod}, with $\varphi=\dif f$. However, since the singularities of $V$ and $\widetilde{\omega}$ are of finite order, we can apply the same reasoning and obtain a similar short exact sequence for the flat Poisson cohomology. Denote by $\Omega_0^{\bullet}(\R^3)$ the space of differential forms on $\R^3$ which are flat at zero. The following holds:

\begin{proposition}\label{ses for flat things}
The flat Poisson complex of $\mathfrak{sl}_2^*(\mathbb{R})$ fits into the short exact sequence:
 \[ 0\to (\dif f\wedge \Omega_0^{\bullet}(\R^3),\dif)\xrightarrow{j_{\dif f}} (\mathfrak{X}^{\bullet}_0(\mathfrak{sl}_2^*(\mathbb{R})),\dif_{\pi})\xrightarrow{p_{\dif f}} (\dif f\wedge \Omega^{\bullet-1}_0(\R^3),\dif)\to 0,
 \]
 where $j_{\dif f}=(-\pi^{\sharp})\circ i_{V}$ and  $p_{\dif f}=e_{\dif f}\circ (-\widetilde{\omega}_b)\circ i_{\dif f}$.
\end{proposition}
\begin{proof}
First, note that $i_V$ is indeed well-defined: if $\eta$ is a flat form at $0$, then $i_V\eta$ extends smoothly at zero and is also flat. The same applies also to the map $\widetilde{\omega}_b$, hence the maps $j_{\dif f}$ and $p_{\dif f}$ are well-defined. They are chain maps because they satisfy this condition away from $0$. In order to show that the sequence is exact, note that also the maps $p_V$ and $j_V$ defined in \eqref{dual_maps} induce maps on flat forms/multi-vector fields. Relations \eqref{dual_maps_relations} (which still hold, because they hold away from the origin) imply that the sequence in the statement is indeed exact. 
\end{proof}

We call the cohomology of the complex 
\begin{equation}\label{ffcomplex}
(\dif f\wedge\Omega^{\bullet}_0(\R^3),\dif)
\end{equation}
the flat foliated cohomology, and denote it by
 \[H^{\bullet}_0(\F,\nu^*).\]
Consider the angular one-form on $\R^3\backslash\{x=y=0\}$ by
\[\dif \theta=\frac{1}{x^2+y^2}(-y\dif x+x\dif y).\]
The proof of the following result will occupy Sections \ref{section flat foli} and \ref{section: analysis}:

\begin{theorem}\label{flat foliated cohomology}
Cohomology classes in $H^{k}_0(\mathcal{F},\nu^*)$ have unique representatives of the form:  
    \begin{itemize}
     \item for $k=0$ 
\[    \chi \dif f, \ \ \textrm{where} \ \ \chi \in C^{flat},\]
     \item for $k=1$
     \[\eta \dif f  \wedge \dif \theta, \ \ \textrm{where}  \  \ \eta\in C^{out},\]
    \item and for $k\geq 2$, $H^{k}_0(\mathcal{F},\nu^*)=0$.
    \end{itemize}
   \end{theorem}
   
\subsection{Higher Poisson cohomology groups}

In this subsection we finish the proof of Theorem \ref{main theorem}. 

Since $H^2_0(\F,\nu^*)=0$, the long exact sequence in cohomology induced by the short exact sequence in Proposition \ref{ses for flat things} yields:\\

\noindent\underline{In degree $0$}: $j_{\dif f}$ induces an isomorphism:
    \[H^{0}_0(\F,\nu^*)\simeq C^{flat}.\]
This isomorphism is simply $j_{\dif f}(\chi \dif f)=\chi$.\\

\noindent\underline{In degree $1$}: $j_{\dif f}$ and $p_{\dif f}$ induce a short exact sequence:
    \begin{equation}\label{ses in coho}
    0\to H^{1}_0(\F,\nu^*)\to H^1_0(\mathfrak{sl}_2^*(\mathbb{R}),\pi)\to H^0_0(\F,\nu^*)\to 0,
    \end{equation}
The map $j_{\dif f}$ acts on the representatives of $H^1_0(\F,\nu^*)$ as follows:
\[j_{\dif f}(\eta \dif f \wedge \dif \theta)=-\eta\pi^{\sharp}(\dif\theta), \ \ \textrm{where}\ \eta\in C^{out}.\]
Since $T|_O=\pi^{\sharp}(\dif \theta)|_O$, the image of $j_{\dif f}$ consists of all the elements
\[[\eta T],\ \ \ \textrm{with}\ \ \eta \in C^{out}.\] 
One other hand, since $\Lie_Nf=1$, note that 
\[p_{\dif f}(\chi N)= \chi\dif f, \ \ \  \textrm{for all } \ \chi\in C^{flat}.\]
Therefore the set consisting of classes of the form $[\chi N]$ is sent by $p_{\dif f}$ onto $H^0(\F,\nu^*)$. Exactness of the sequence \eqref{ses in coho} and Theorem \ref{flat foliated cohomology}, imply that elements in $H^1_{0}(\mathfrak{sl}_2^*(\mathbb{R}),\pi)$ can be uniquely represented as
\[\eta T+\chi N,\]
with $\eta\in C^{out}$ and $\chi\in C^{flat}$. By Corollaries 
\ref{formal cohomology sl2} and \ref{ses smooth formal}, \[H^1(\mathfrak{sl}_2^*(\mathbb{R}),\pi)=H^1_{0}(\mathfrak{sl}_2^*(\mathbb{R}),\pi);\] thus we obtain the description of $H^1(\mathfrak{sl}_2^*(\mathbb{R}),\pi)$ from Theorem \ref{main theorem}. \\

\noindent\underline{In degree $2$}: $p_{\dif f}$ induces an isomorphism:
\[H^2_0(\mathfrak{sl}_2^*(\mathbb{R}),\pi)\simeq H^1_0(\F,\nu^*).\]
We note that, for all $\eta\in C^{out}$,
\[p_{\dif f}(\eta N\wedge T)=-\eta \dif f\wedge \dif \theta.\]
Hence, reasoning as in the previous case, we obtain the description of the second Poisson cohomology group $H^2(\mathfrak{sl}_2^*(\mathbb{R}),\pi)$ from Theorem \ref{main theorem}.\\

\noindent\underline{In degree $3$}: we have that:
\[H^3_0(\mathfrak{sl}_2^*(\mathbb{R}),\pi)= 0.\]
By using again Corollaries \ref{formal cohomology sl2} and \ref{ses smooth formal}, we obtain the description from Theorem \ref{main theorem} of  $H^3(\mathfrak{sl}_2^*(\mathbb{R}),\pi)$. 

\section{Flat foliated cohomology}\label{section flat foli}

In this section we reduce the proof of Theorem \ref{flat foliated cohomology} to two technical results, which will be proven in Section \ref{section: analysis}.

\subsection{Averaging over $S^1$}

In order to compute the flat foliated cohomology, it will be more convenient to work with $S^1$-invariant forms, where we consider the natural action of $S^1$ on $\R^3$ by rotations around the $z$-axis. Since $ f$ is $S^1$-invariant, and $0$ is fixed by the action, the invariant part of \eqref{ffcomplex} forms a subcomplex, denoted:
\begin{equation*}
(\CC^{\bullet},\dif):=(\dif f\wedge \Omega^{\bullet}_0(\R^3)^{S^1},\dif).
\end{equation*}
Note that averaging operator 
\[Av:(\dif f\wedge \Omega^{\bullet}_0(\R^3),\dif)
\to (\CC^{\bullet},\dif),\]
\[Av(\alpha)=\frac{1}{2\pi}\int_0^{2\pi} \big(e^{i\theta}\big)^*\alpha\dif \theta,\]
is a chain map and a projection onto $\CC^{\bullet}$. 

\begin{lemma}\label{lemma:averaging}
The map $Av$ induces an isomorphism in cohomology.
\end{lemma}
\begin{proof}
Let $\partial_{\theta}=x\partial_{y}-y\partial_x$ be the rotational vector field generating the $S^1$-action. Then, for all $\alpha\in \dif f\wedge \Omega^{\bullet}_0(\R^3)$,
\begin{align}\label{eq:pull back}
(e^{i\theta})^*\alpha-\alpha&=\int_0^{\theta}\frac{\dif}{\dif t}(e^{it})^*\alpha\dif t=\int_0^{\theta}\Lie_{\partial_{\theta}}(e^{it})^*\alpha\dif t\\\nonumber
&=\dif\int_0^{\theta}i_{\partial_{\theta}}(e^{it})^*\alpha\dif t+ \int_0^{\theta}i_{\partial_{\theta}}(e^{it})^*\dif \alpha\dif t.
\end{align}
Integrating this equation from $0$ to $2\pi$, we obtain
\begin{equation}\label{eq:averaging}
Av(\alpha)-\alpha=\dif \circ h (\alpha)+h\circ \dif (\alpha),
\end{equation}
where $h$ denotes the homotopy operator: 
\[h:\dif f\wedge \Omega^{\bullet}_0(\R^3)
\to \dif f\wedge \Omega^{\bullet-1}_0(\R^3),\]
\[h(\alpha)=\frac{1}{2\pi}\int_0^{2\pi}\dif \theta \int_0^{\theta}
i_{\partial_{\theta}}(e^{it})^*\alpha \dif t=
\frac{1}{2\pi}\int_0^{2\pi}(2\pi-t)i_{\partial_{\theta}}(e^{it})^*\alpha \dif t.\]
The homotopy relation \eqref{eq:averaging} implies the statement.
\end{proof}

\subsection{Retraction to the ``cohomological skeleton''}

In order to calculate the cohomology of $\CC^{\bullet}$, we use a retraction onto the set 
\begin{align*}
X:=\{ x=y=0\} &\cup \{z=0\}
\end{align*}
along the leaves of the foliation. We think about $X$ as a ``cohomological skeleton'' of the singular foliation. The leaves in $I$ are diffeomorphic to $\R^2$ and $X$ intersects them exactly in one point, and the leaves $O$ are diffeomorphic to $S^1\times \R$, and $X$ intersects these in one circle. On the other hand, $X$ does not intersect the two leaves in the cone, but, as we will see, these will not contribute to the flat cohomology. Define the retraction as follows:
\begin{align}\label{retract}
p_{X} (x,y,z)&:=
   \begin{cases}
    (\frac{x}{r}\sqrt{f},\frac{y}{r}\sqrt{f},0)&\text{ on } O\\
    (0,0,0)&\text{ on }Z\\
    (0,0,\sign(z)\sqrt{-f})&\text{ on }I
\end{cases}
\end{align}
Note that $p_X$ preserves $\R^3\backslash Z$, and it satisfies:
\[p_X(\R^3)=X, \ \ \textrm{and}\ \ p_X|_{X}=\id_X.\]
Also, note that $p_X$ is continuous on $\R^3$, it is smooth on $\R^3\backslash Z$, but it is not smooth on $Z$. However, since we are working with forms that are flat at $0$, the following holds:

\begin{lemma}\label{infinite time pullback}
For every $\alpha \in \Omega_{0}^{\bullet}(\R^3)$, the form $p_{X}^*(\alpha)|_{\R^3\backslash Z}$ extends to a smooth form on $\R^3$, which satisfies $p_{X}^*(\alpha)\in \Omega_{0}^{\bullet}(\R^3)$. \end{lemma}

The proof of this result is given at the end of Subsection \ref{subsection: homotopy operators}.

We have that $p_X^*$ is $S^1$-equivariant, it commutes with $\dif$ and with $e_{\dif f}$; these properties hold, because they are closed and they hold on $\R^3\backslash Z$. Therefore, $p_X^*$ induces a chain map:
\[p_X^*:(\CC^{\bullet},\dif)\to (\CC^{\bullet},\dif).\]
In the next subsection, we will show that this is an isomorphism in cohomology, more precisely:
\begin{proposition}\label{lemma: homotopies}
There are linear maps 
\[h:\CC^{\bullet}\to \CC^{\bullet-1}\]
which satisfy the homotopy relation:
\[p_X^*(\alpha)-\alpha=\dif \circ h (\alpha)+h\circ \dif (\alpha).\]
\end{proposition}

Hence, $p_X^*(\CC^{\bullet})$ has the same cohomology as $\CC^{\bullet}$. However:

\begin{lemma}\label{lemma: closed}
For all $\alpha\in \CC^{\bullet}$, we have that \[\dif p_X^*(\alpha)=0.\] 
\end{lemma}
\begin{proof}
Let $\alpha\in \CC^{k}$. On $I$, we have that 
\[\dif p_X^*(\alpha)|_I=p_X^*(\dif \alpha|_{\{0\}\times \R})=0,\]
because $\dif \alpha$ is at least a 2-form. Similarly, if $k>0$, also on $O$ we have that:
\[\dif p_X^*(\alpha)|_O=p_X^*(\dif \alpha|_{\R^2\times\{0\}})=0.\]
For $k=0$, write $\alpha=\chi\dif f$. Note that $f|_{\R^2\times 0}=r^2$. Since $\chi$ is $S^1$-invariant, we can write $\chi|_{\R^2\times\{0\}}=h(r^2)$, for some flat function $h$; hence $\alpha|_{\R^2\times \{0\}}=h(r^2)\dif r^2$, and so $\dif \alpha|_{\R^2\times \{0\}}$=0. 
\end{proof}

We are ready now to prove Theorem \ref{flat foliated cohomology}.
\begin{proof}[Proof of Theorem  \ref{flat foliated cohomology}]
By Lemma \ref{lemma:averaging} and Proposition \ref{lemma: homotopies}, the subcomplex $p_X^*(\CC^{\bullet})$ computes $H^{\bullet}_0(\F,\nu^*)$. By Lemma \ref{lemma: closed}, the differential on this subcomplex is trivial, and so every class in $H^{\bullet}_0(\F,\nu^*)$ has a unique representative in $p_X^*(\CC^{\bullet})$. Thus it suffices to determine the image of $p_X^*$.\\
\noindent\underline{In degree $0$}: this follows from Proposition \ref{ses for flat things}.\\
\noindent\underline{In degree $1$}: let $\alpha\in \CC^1$. As in the proof of Lemma \ref{lemma: closed}, we have that $p_X^*(\alpha)|_I=0$. Note that $\dif f\wedge \dif\theta|_{\R^2\times\{0\}}=2\dif x\wedge\dif y$, therefore we can write uniquely $\alpha|_{\R^2\times \{0\}}=g\dif f\wedge \dif\theta|_{\R^2\times\{0\}}$, and since $\alpha$ is $S^1$-invariant and flat at $0$, we can further decompose $g=h(r^2)$, where $h\in C^{\infty}(\R)$ with $\supp h \subset [0,\infty)$. Thus, for $\eta=h\circ f\in C^{out}$, we obtain: 
\[p_X^*(\alpha)= \eta \dif f\wedge \dif \theta.\]
\noindent\underline{In degree $2$}: as in the proof of Lemma \ref{lemma: closed}, $p_X^*(\alpha)=0$ for any $\alpha\in \CC^2$.
\end{proof}




\subsection{Homotopy operators}\label{subsection: homotopy operators}
\begin{wrapfigure}{r}{5cm}
\vspace{-20pt}
 \includegraphics[scale=0.6]{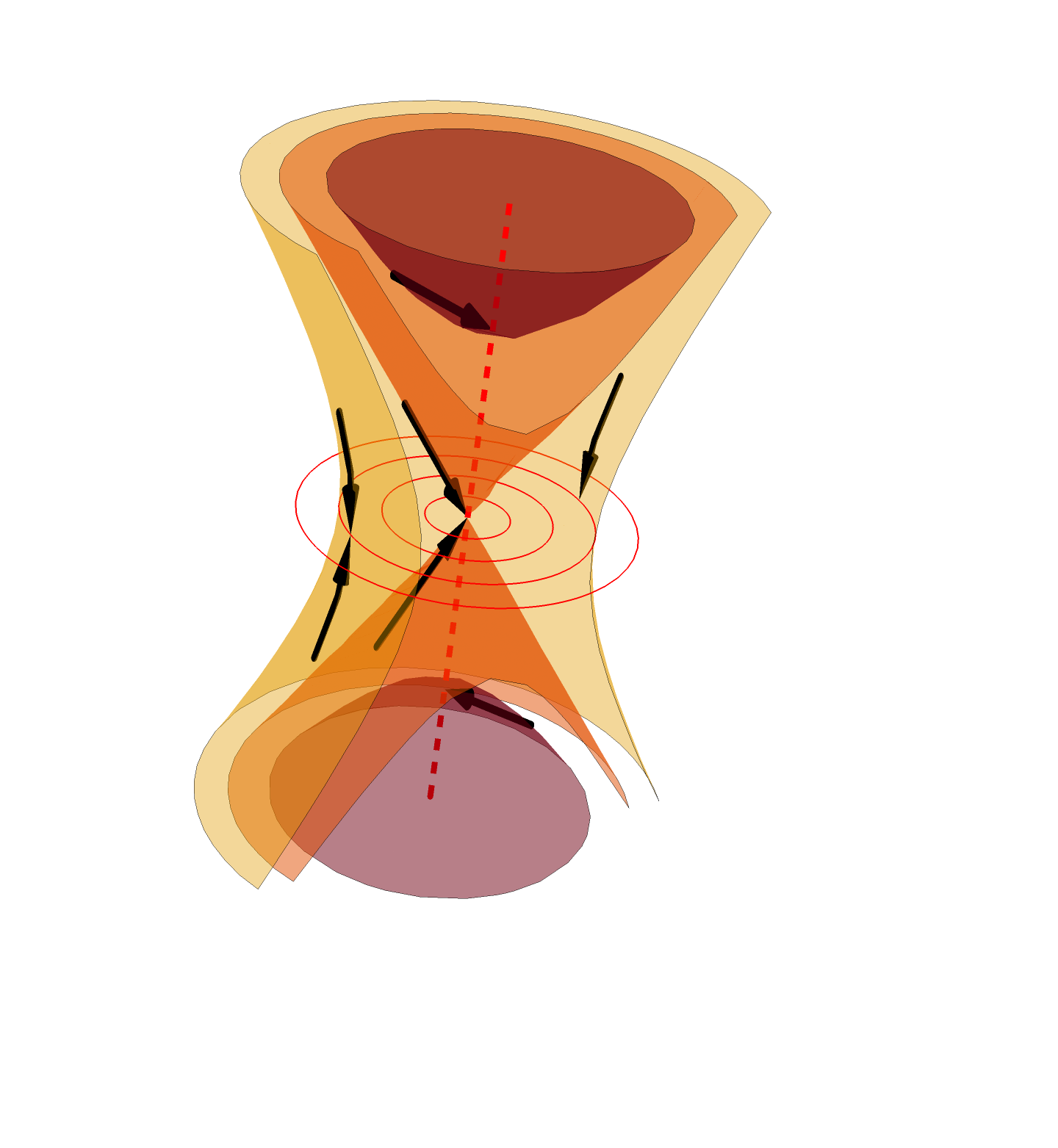}
 \vspace{-60pt}
 \caption{Flow of $W$ along the leaves to the ``cohomological skeleton''}
  \vspace{-80pt}
 \end{wrapfigure}

In order to construct the homotopy operators from Proposition \ref{lemma: homotopies}, we build a foliated homotopy between the identity map and the retraction $p_X$. We will do this using the flow of the vector field:
  \begin{align}\label{W}
   W=& -r^2z\pi^{\sharp}(\dif \theta)=-r^2 z\partial_z -z^2 r\partial_r\\
   =&-z^2x\partial_x -z^2y\partial_y -(x^2+y^2)z\partial_z\nonumber
  \end{align}
 where $r=\sqrt{x^2+y^2}$. Note that $W$ has the following properties:
 \begin{itemize}
     \item $W$ vanishes precisely on $X$,
     \item $W$ is tangent to the foliation:\\ $\Lie_Wf=0,$
     \item $W$ is $S^1$-invariant,
     \item $\Lie_W(\theta)=\dif \theta(W)=0$,
     \item $\Lie_W(R^2)=-4r^2z^2\leq 0$,\\ where $R^2=r^2+z^2.$
\end{itemize}
In particular $i_W$ preserves $\CC^{\bullet}$. 
The last property implies that the flow lines starting in a point inside the closed ball $\overline{B}_n(0)$ are trapped inside that ball; hence the flow is defined for all positive time, and will be denoted by:
  \begin{equation*}
  \begin{array}{cccc}
   \phi:&[0,\infty)\times \R^3 &\to&\R^3\\
   &(t,x,y,z)&\mapsto&\phi_t(x,y,z).
  \end{array}
  \end{equation*}
The above properties of $W$ imply that:
\begin{itemize}
\item $\phi_t$ fixes $X$;
\item $\phi_t^*(f)=f$; 
\item $\phi_t$ is $S^1$-equivariant;
\item $\phi_t^*(\dif \theta)=\dif \theta$.
\end{itemize}
In particular, $\phi_t^*$ preserves the complex $(\CC^{\bullet},\dif)$. By a similar calculation as \eqref{eq:pull back}, for all $\alpha\in \CC^{\bullet}$, we have that 
\begin{equation}\label{eq: at time t}
\phi_t^*(\alpha)-\alpha=\dif \circ h_t(\alpha)+h_t\circ \dif(\alpha),    
\end{equation}
where
\[h_t:\CC^{\bullet}\to \CC^{\bullet-1}, \ \ h_t(\alpha)=\int_0^{t}i_W\phi_s^*(\alpha)\dif s.\]
In order to prove Proposition \ref{lemma: homotopies} we will take the limit as $t\to\infty$ in \eqref{eq: at time t}. In the following subsection we will give explicit formulas for $\phi_t$, which imply the point-wise limit: 
\begin{align}\label{infinity now}
\lim_{t\to\infty}\phi_t(x,y,z)=p_X(x,y,z), \ \ \ \forall\  (x,y,z)\in \R^3.
\end{align}
The following results are much more involved, and their proofs will occupy the last section of the paper:
\begin{lemma}\label{convergence of the flow}
On $\R^3\backslash Z$, we have that \[\lim_{t\to\infty}\phi_t=p_X\] 
with respect to the compact-open $C^{\infty}$-topology.
\end{lemma}

\begin{lemma}\label{well-defined}
For any $\alpha\in \Omega_{0}^{\bullet}(\R^3)$, the following limit exists:
   \begin{align}\label{limit h}
    h(\alpha):=\lim_{t \to \infty}{h_{t}(\alpha)}\in \Omega_{0}^{\bullet-1}(\R^3),
   \end{align}
with respect to the compact-open $C^{\infty}$-topology. 
\end{lemma}

Recall that the existence of a limit with respect to the compact-open $C^{\infty}$-topology means that all partial derivatives converge uniformly on compact subsets; more details are given in the following section. 

The results above suffice to complete our proofs:

\begin{proof}[Proofs of Lemma \ref{infinite time pullback} and Proposition \ref{lemma: homotopies}]
Since the limit \eqref{limit h} is uniform on compact subsets with respect to all $C^k$-topologies, $h$ satisfies
   \begin{align*}
    \dif h(\alpha)&=\lim_{t \to \infty}{\dif h_{t}(\alpha)}.
   \end{align*}
From \eqref{eq: at time t}, we obtain that for any $\alpha\in \Omega_0^{\bullet}(\R^3)$
\[\lim_{t\to\infty}\phi_t^*(\alpha)=\alpha+\dif\circ h(\alpha)+h\circ \dif (\alpha)\]
holds for the compact-open $C^{\infty}$-topology. On the other hand, on $\R^3\backslash Z$, $\lim_{t\to\infty}\phi_t=p_X$, and since this limit is also with respect to the compact-open $C^{\infty}$-topology, we have that 
\[\lim_{t\to\infty}\phi_t^*(\alpha)|_{\R^3\backslash Z}=p_X^*(\alpha)|_{\R^3\backslash Z}.\]
Therefore, $p_X^*(\alpha)|_{\R^3\backslash Z}$ extends to a smooth form on $\R^3$. This implies Lemma \ref{infinite time pullback} and the equation:
\[p_X^*(\alpha)-\alpha=\dif\circ h(\alpha)+h\circ \dif (\alpha).\]
Finally, since $h_t$ is $S^1$-equivariant and commutes with $e_{\dif f}$, and these conditions are closed, we have that $h(\CC^{\bullet})\subset \CC^{\bullet-1}$. Thus, the above relation holds on $\CC^{\bullet}$, and so we obtain also Proposition \ref{lemma: homotopies}.
\end{proof}

\subsection{Explicit formula for the flow}\label{subsection:explicit flow}
 
Recall that in cylindrical coordinates $W=-(r^2z\partial_z+z^2r\partial_r)$. Therefore, its flow $\phi_t(r,\theta,z)=(r_t,\theta_t,z_t)$ satisfies 
\[r'_t=-r_tz^2_t,\ \  \ \ \theta'_t=0,\ \ \ \ z'_t=-z_tr^2_t.\]
As remarked before, $\theta_t=\theta$. The above system is equivalent to
\begin{equation*}
(r^2_t)'=-2r^2_tz^2_t,\ \  \ \ (z^2_t)'=-2r^2_tz^2_t.
\end{equation*}
Note that $(r^2_t-z_t^2)'=0$, hence, as remarked before, $f=r^2-z^2$ is constant along the flow lines. Therefore, the system above is equivalent to a single ODE in $R_t^2=r^2_t+z^2_t$. Solving this ODE, we obtain the explicit formulas:
\begin{align}\label{explicit_flow}
r_t&=r\sqrt{\frac{r^2-z^2}{r^2-z^2e^{-2t(r^2-z^2)}}}=\frac{r}{\sqrt{1+tz^2\kappa(tf)}},\\ 
z_t&=z\sqrt{\frac{z^2-r^2}{z^2-r^2e^{2t(r^2-z^2)}}}=\frac{z}{\sqrt{1+tr^2\kappa(-tf)}}, \nonumber
\end{align}
where we have denoted by $\kappa$ the following smooth function:
  \begin{equation*}
   \begin{array}{cccc}
   \kappa:&\R&\to&\R_{\ge 0}\\
   &x&\mapsto&\int_0^2e^{-sx}\dif s=\frac{1-e^{-2x}}{x}
   \end{array}
\end{equation*}
These formulas give the point-wise limit $\lim_{t\to\infty}\phi_t=p_X$ claimed in \eqref{infinity now}: 
\begin{equation*}
\lim_{t\to\infty}(r_t,z_t)=\begin{cases}
    (\sqrt{f},0)&\text{ if }\  0\leq f\\
    (0,\sign(z)\sqrt{-f})&\text{ if }\  f\leq 0
   \end{cases}.
\end{equation*}

In Cartesian coordinates, we obtain:
\begin{lemma}\label{flow lines}
For $t\geq 0$, the flow $\phi_t(x,y,z)=(x_t,y_t,z_t)$ of $W$ is given by 
\begin{align}\label{flow in Cartesian}
x_t&= \frac{x}{\sqrt{1+tz^2\kappa(tf)}},\nonumber \\
y_t&= \frac{y}{\sqrt{1+tz^2\kappa(tf)}},\\
z_t&= \frac{z}{\sqrt{1+tr^2\kappa(-tf)}}.\nonumber 
\end{align}
\end{lemma}

\section{The analysis}\label{section: analysis}
  
In this last section we prove Lemmas \ref{convergence of the flow} and 
\ref{well-defined}.

\subsection{Partial derivatives, Leibniz rule, chain rule}

Denote the partial derivative corresponding to a multi-index \[a=(a_1,\ldots,a_n)\in \N^n \ \ \textrm{by}\ \  D^{a}: = \frac{1}{a_1!}\partial^{a_1}_{x_1}\ldots \frac{1}{a_n!}\partial^{a_n}_{x_n}.\] 
We will often use the general Leibniz rule:
\[D^a(f_1\cdot \ldots\cdot f_k)=\sum_{a^1+\ldots+a^k=a}D^{a^1}(f_1)\cdot\ldots \cdot D^{a^k}(f_k),\]
for $f_1,\ldots,f_k\in C^{\infty}(\R^n)$, and the general chain rule:
\[D^a(g(f_1,\ldots,f_k))=\sum_{1\leq |b|\leq |a|}D^b(g)(f_1,\ldots,f_k){\sum}' \prod_{i=1}^k\prod_{j=1}^{b_i} D^{a^{ij}}(f_i),\]
where $b=(b_1,\ldots,b_k)$ and $\sum'$ is the sum over all non-trivial decompositions  \[a=\sum_{i=1}^k\sum_{j=1}^{b_i}a^{ij}, \ \ \  a, a^{ij}\in \N^n.\]

\subsection{Limits of families of smooth functions}\label{subsection: limits}

We discuss some standard facts about the existence of limits of families of smooth functions. 

Let $U\subset \R^n$ be an open set. For a compact subset $K\subset U$, we define the corresponding $C^k$-semi-norm on $C^{\infty}(U,\R^m)$ as:
\[\norm{F}^K_k:=\sum_{|a|\leq k}\sup_{x\in K}\left|D^aF(x)\right|,\]
where $|a|=a_1+\ldots+a_n$. The semi-norms 
\[\big\{\, \norm{\cdot}_k^K\ | \ k\geq 0,\  K\subset U \big\},\]
endow $C^{\infty}(U,\R^m)$ with the structure of a Fr\'echet space, and the resulting topology is called the compact-open $C^{\infty}$-topology. 

Consider a family $F_t\in C^{\infty}(U,\R^m)$, defined for $t\geq t_0$. Assume that, for each compact $K\subset U$ and each $k\geq 0$, we find a function 
\[l^K_{k}:[t_0,\infty)\to [0,\infty),\]
such that
\[\forall \  s\geq t\ :\ \ \norm{F_s-F_t}^K_{k}\leq l^K_{k}(t) \ \ \ \ \textrm{and} \ \ \ \ \lim_{t\to \infty}l_k^K(t)=0.\]
Then, since $C^{\infty}(U,\R^m)$ is a Fr\'echet space, the limit $t\to \infty$ exists: \[F_{\infty}:=\lim_{t\to \infty}F_{t}\in C^{\infty}(U,\R^m).\]
So all partial derivatives of $F_t$ convergence uniformly on all compact subsets to those of $F_{\infty}$.

Assume further that $F_t$ is smooth also in $t$. Writing  
\[D^a(F_t-F_s)=\int_s^tD^a(F'_h)\dif h,\]
we obtain that 
\[\norm{F_s-F_t}^K_{k}\leq \int_s^t\norm{F'_h}_k^K\dif h.\]
So, if we find a function 
\[b^K_{k}:[t_0,\infty)\to [0,\infty),\]
such that
\begin{equation}\label{conditions for limit}
\forall \  t\geq t_0\ :\ \ \norm{F'_t}^K_{k}\leq b^K_{k}(t)\ \ \ \ \textrm{and} \ \ \ \ \int_{t_0}^{\infty}b_k^K(t)\dif t<\infty,
\end{equation}
then we can conclude that $\lim_{t\to\infty} F_t$ exists. We will use this criterion in the following subsections. 

\subsection{Polynomial-type estimates}\label{subsection: flow estimates}

In the following subsections, we will prove several inequalities, and in order to keep track of what is essential, we discuss here the nature of these estimates. The following two families of functions play a key role:
\[g_t, \overline{g}_t:\R^3\to (0,1],\ \ \ t\in [0,\infty),\]
\[g_t:=\frac{1}{\sqrt{1+tz^2\kappa(tf)}},\ \ \  \  \overline{g}_t:=\frac{1}{\sqrt{1+tr^2\kappa(-tf)}}.\]
These functions appeared in the explicit formula for the flow $\phi_t=(x_t,y_t,z_t)$
\[x_t=xg_t, \ \ \ y_t=yg_t,\ \ \ r_t=rg_t\ \ z_t=z\overline{g}_t.\]
Note that they satisfy the following property: 
\begin{equation}\label{g_t}
g_t=\overline{g}_te^{tf},
\end{equation}

As discussed in the previous subsection, given a smooth family 
$F_t\in C^{\infty}(\R^3)$, $t\geq 0$, in order to show that $\lim_{t\to \infty}F_t$ exists, we need to estimate the partial derivatives of $F'_t$. We will find bounds which are polynomials in the variables $R$, $t$, $g_t$ and $\overline{g}_{t}$, and obtain inequalities of the form:
\begin{equation}\label{polynomial bounds}
|D^a F'_t|\leq C\sum R^d\, t^k\, g_t^p\, \overline{g}_{t}^q, \ \ \ \ \ \ \forall\ (x,y,z)\in \R^3, \ \ t\geq 0,
\end{equation}
where $C>0$ is a constant, $R=\sqrt{x^2+y^2+z^2}$, and the sum is over a finite set of degrees $(d,k,p,q)$. Which degrees actually appear in this sum will play a crucial role. Namely, note first that, by \eqref{g_t}, $g_t$ is rapidly decreasing on $I$ and $\overline{g}_t$ is rapidly decreasing on $O$. So, for $p>0$ and $q>0$, $R^d\, t^k\,  g^p_t\,\overline{g}_t^q$ goes rapidly to zero away from $Z$. However, higher exponents $d$, $p$ and $q$ are needed to obtain estimates as in \eqref{conditions for limit} also along $Z$. For this we will use the following:

\begin{lemma}\label{lemma: finite integrals}
For $p>0$ and $q>0$ there is $C=C(p,q)$ such that:
\[g_t^p\, \overline{g}_t^q\leq C R^{-(p+q)}t^{-\frac{p+q}{2}},\]
for all $t>0$ and all $(x,y,z)\neq 0\in \R^3$.
\end{lemma}

\begin{proof}
First we prove that, for $0\leq \epsilon \leq 1$, $0\le s$, $0\leq v$, the following holds:
\begin{equation}\label{bounded function}
   \frac{\epsilon}{1+v\kappa(s)} e^{- 2\epsilon s}\le \frac{1}{1+2(v+s)}.
\end{equation}
This is equivalent to 
\[l(v,s):=(s+v(1-e^{-2s}))e^{2\epsilon s}-\epsilon s(1+2(v+s))\geq 0.\]
Since $l$ is linear in $v$, we need to check that $l(0,s)\geq 0$ and $\partial_vl(0,s)\geq 0$:
\begin{align*}
l(0,s)&=s\big(e^{2\epsilon s}-\epsilon (1+2s)\big)\\
 &\geq
s\big(1+2\epsilon s-\epsilon (1+2s)\big)\\
&=s(1-\epsilon)\geq 0,\\
\partial_vl(0,s)&=(1-e^{-2s})e^{2\epsilon s}-2\epsilon s\\
&=e^{2\epsilon s}-e^{-2(1-\epsilon)s}-2\epsilon s\\
&\geq e^{2\epsilon s}-1-2\epsilon s\geq 0.
\end{align*}

Next, we prove the statement in the case $0\leq |z|\leq r$, the other case follows similarly. Using \eqref{g_t}, we write:
\[g_t^p\,\overline{g}_t^q=g_t^{p+q}e^{-tqf}=\left(g_t^2e^{-t\frac{2q}{p+q}f}\right)^{\frac{p+q}{2}}.\]
By applying \eqref{bounded function} with $s:=t(r^2-z^2)$, $v:=tz^2$ and $\epsilon:=\frac{q}{p+q}$, we obtain:
\[
g_t^2e^{-t\frac{2q}{p+q}f}=\frac{e^{-2\epsilon s}}{1+v\kappa(s)}\leq \epsilon^{-1}\frac{1}{1+2(s+v)}= \frac{p+q}{q}\frac{1}{1+2tr^2},
\]
and so: 
\[g_t^p\, \overline{g}_t^q\leq \left(\frac{p+q}{q}\right)^{\frac{p+q}{2}} (1+2tr^2)^{-\frac{p+q}{2}}.\]
Using that $t R^2\leq (1+2 t r^2)$ we obtain the inequality from the statement.
\end{proof}

For the estimates that will follow, we introduce polynomials $\sigma_{k,l}$ defined for $k\in \N$ and $l\in \Z$ as the sum of all monomials $t^jR^d$ with  
\[d-2j=l \ \  \textrm{and}\ \ 0\leq j\leq k,\]
or, in a closed formula: 
\[\sigma_{k,l}=\sum_{m\leq j\leq k} t^jR^{2j+l}, \ \ m:=\mathrm{max}\big(0,-l/2 \big),\]
and we set $\sigma_{k,l}=0$ if $m> k$ and $\sigma_{0,0}=1$.

By comparing terms, the following is immediate: 
\begin{equation}\label{submult}
\sigma_{k,l}\, \sigma_{k',l'}\leq C \sigma_{k+k', l+l'},
\end{equation}
for some constant $C=C(k,k',l,l')$. In particular:
\begin{equation*}
R^d\sigma_{k,l}=\sigma_{0,d}\sigma_{k,l}\leq \sigma_{k, l+d}.
\end{equation*}
We will use these inequalities later on. 

\subsection{Estimates for $g_t$ and $\overline{g}_t$}

Here we will evaluate the partial derivatives of $g_t$ and $\overline{g}_t$. Let 
\[g(u,v):=\frac{1}{\sqrt{1+v\kappa(u-v)}}, \ \ \ \  u,v\geq 0.\]
First, we prove an intermediate result about the function $g$.

\begin{lemma}\label{lemma gg}
For every $a\in \N^2$ there exists $C=C(a)$ such that:
\begin{equation}\label{gg}
    \lvert D^a g(u,v)\rvert \leq C g(u,v), \ \ \forall \ u,v\geq 0.
\end{equation}    
\end{lemma}
\begin{proof}
Consider the function:
\[\iota(u,v):=1+v\kappa(u-v).\]
First note that 
   \begin{align*}
    \lvert \kappa ^{(k)}(x)\vert=\int_0^2s^ke^{-xs}\dif s\leq 2^k\kappa(x).
   \end{align*}
Using this, that $\kappa(s)\leq 2$ for $0\leq s$, and the Leibniz rule, one finds for any $b\in \N^2$ a constant $C=C(b)>0$ such that, for all $0\leq v\leq u$:
  \begin{align*}
   \lvert D^b (\iota(u,v))\rvert &\le C\iota(u,v).
  \end{align*}
Next, using that 
\[\left(x^{-\frac{1}{2}}\right)^{(k)}=c_kx^{-k-\frac{1}{2}}, \ \ c_k=(-1)^k\ \frac{2k-1}{2}\cdot \frac{2k-3}{2}\cdot \ldots \cdot \frac{1}{2},\]  
and the chain rule, we obtain the following estimate for $0\leq v\leq u$:  
\begin{align*}
    \lvert D^a \iota^{-\frac12}\rvert &=\lvert \sum_{1\le k \le |a|}{ c_k\iota^{-k-\frac{1}{2}}\sum_{a^1+\dots +a^k=a}{D^{a^1}(\iota)\dots D^{a^k}(\iota)}}\rvert\leq C\iota^{-\frac12}.
\end{align*} 
Since $g=\iota^{-\frac{1}{2}}$, we obtain \eqref{gg} on the domain $0\leq v\leq u$. In order to prove the estimate also for $0\leq u\leq v$, consider the function $\overline{g}(u,v):=g(v,u)$. Clearly, $\overline{g}$ satisfies the version of inequality \eqref{gg} for $0\leq u\leq v$. Note the following relation (which is equivalent to \eqref{g_t}):
\[g(u,v)=e^{v-u}\overline{g}(u,v).\]
Using this, we obtain \eqref{gg} also for $0\leq u\leq v$:
\begin{align*}
|D^a(g(u,v))|&=|\sum_{a^1+a^2=a}D^{a^1}(e^{v-u})D^{a^2}(\overline{g}(u,v))|\\
&\leq Ce^{v-u}\overline{g}(u,v)=Cg(u,v).\qedhere
\end{align*}
\end{proof}

We provide now estimates for the partial derivatives of $g_t$ and $\overline{g}_t$. 

\begin{lemma}\label{lemma:bounds on g}
For $a\in \N^3$, with $|a|=k$, there is $C=C(a)$ such that
\[|D^ag_t|\leq C \sigma_{k,-k}g_t,\ \ \ \ \textrm{and} \ \ \ \  
   |D^a\overline{g}_t|\leq C \sigma_{k,-k}\overline{g}_t.\]
\end{lemma}

\begin{proof}
We have that $g_t=g(tr^2,tz^2)$. Therefore, by the chain rule: 
 \begin{alignat*}{2}
    \lvert D^a g_t\rvert &=\; && \lvert \sum_{1\le |b| \le |a|}{D^b(g)(tr^2,tz^2){\sum}' {\Pi_{i=1}^{b_1}{D^{a^{1i}}(tr^2)}\Pi_{i=1}^{b_2}{D^{a^{2i}}(tz^2)}}}\rvert,
   \end{alignat*}
where $b=(b_1,b_2)$ and the sum ${\sum}'$ is over all decompositions 
\[a=a^{11} +\dots+ a^{1b_1}  +a^{21} +\dots +a^{2b_2},\] with $a^{ji}\in \N^3$ and $1\leq |a^{ji}|$. Note that $|a^{ji}|\le 2$ since otherwise $D^{a^{1i}}(tr^2)$ and $D^{a^{2i}}(tz^2)$ are zero, respectively. Hence we have
   \begin{align*}
    b_j &\le \sum_{i=1}^{b_j}{\lvert a^{ji}\rvert} \le 2b_j,
   \end{align*}
   and therefore
   $\lvert b \rvert  \le k \le 2\lvert b\rvert$, which means $\frac{k}{2}\le \lvert b\rvert \le k$. Moreover, 
   \begin{align*}
    |D^{a^{1i}}(tr^2)|\le 2tR^{2-\lvert a^{1i}\rvert}
   \end{align*}
   and similarly for $tz^2$.
Thus we obtain the estimate:
\[|{\Pi_{i=1}^{b_1}{D^{a^{1i}}(tr^2)}\Pi_{i=1}^{b_2}{D^{a^{2i}}(tz^2)}}|\leq C t^{|b|}R^{2|b|-k}.\]
Using now the previous lemma, we the first inequality: \[\left\lvert D^a g_t \right\rvert \leq C\sum_{k/2\leq j\leq k}t^jR^{2j-k}g_t=C \sigma_{k,-k}g_t.\]
The statement for $\overline{g}_t$ is proven similarly. 
\end{proof}

\subsection{Estimates on the flow (proof of Lemma \ref{convergence of the flow})}

Next, we estimate the partial derivatives of the flow:

\begin{lemma}\label{lemma:bounds on the flow}
For $a\in \N^3$, with $k=|a|$, there is $C=C(a)$ such that:
\[|D^ax_t|\leq C \sigma_{k,1-k} g_t ,\ \ \ \
|D^ay_t|\leq C \sigma_{k,1-k} g_t , \ \ \ \ 
|D^az_t|\leq C \sigma_{k,1-k}\overline{g}_t.\]
\end{lemma}

\begin{proof}
Recall that $x_t=xg_t$. Therefore:
\[D^a(x_t)=xD^a(g_t)+D^{b}(g_t),\]
where $b=a-(1,0,0)$ and $D^{b}=0$ if $b\notin\N^3$. Using that $|x|\leq R$ and \eqref{submult}, we obtain
\[\left\lvert D^a x_t \right\rvert \leq 
C\big(R\sigma_{k,-k}+\sigma_{k-1,1-k}\big)g_t\leq 
C\sigma_{k,1-k}g_t. \]
The other two estimates are proven in the same way. 
\end{proof}

We are now ready to show convergence of the flow away from $Z$: 

\begin{proof}[Proof of Lemma \ref{convergence of the flow}]
It suffices to consider compact sets of the form:
\[K_n^{\epsilon}:=\overline{B}_n(0)\cap \{|f|\geq \epsilon\}.\]
By the discussion in Subsection \ref{subsection: limits} we need to bound the partial derivatives of $\phi'_t := \frac{\dif}{\dif s}\phi_s |_{s=t}$ on $K_{n}^{\epsilon}$ by a positive integrable function. Since $\phi_t$ is the flow of $W=-z^2x\partial_x -z^2y\partial_y -(x^2+y^2)z\partial_z$, we have that 
\begin{align*}
x'_t&=-z_t^2x_t,\\
y'_t&=-z_t^2y_t,\\
z'_t&=-(x_t^2+y_t^2)z_t.
\end{align*}
Note that, for $R\leq n$ and $1\leq t$, there is $C=C(n,i,l)$ such that 
\[\sigma_{i,l}\leq C t^i.\]
Therefore, using the Leibniz identity and Lemma \ref{lemma:bounds on the flow}, we find for any $a\in \N^3$ a constant $C=C(a)$ such that, for $t\geq 1$,
the following hold on $\overline{B}_n(0)$
   \begin{align*}
    \lvert D^a x'_t\rvert &\le C\, t^{k}\, g_t\, \overline{g}^2_t,\\
    \lvert D^a y'_t\rvert &\le C\, t^{k}\, g_t\, \overline{g}^2_t,\\
    \lvert D^a z'_t\rvert &\le C\, t^{k}\, g^2_t\, \overline{g}_t.
   \end{align*}
Next, note that \eqref{g_t} and $g_t\leq 1$ give:
\[g_t\overline{g}_t=e^{-tf}g_t^2\leq e^{-ft},\]
and similarly, by exchanging their role, we obtain: 
\[g_t\overline{g}_t=e^{tf}\overline{g}_t^2\leq e^{ft}.\]
Thus, the following estimate holds:
\[g_t\overline{g}_t\leq e^{-t|f|}.\]
 Therefore, we obtain for $t\geq 1$:
\[\norm{\phi_t'}_k^{K_{n}^{\epsilon}} \leq C\, t^{k}\, e^{-\epsilon t}.\]
Since the right hand side is integrable, the conclusion follows.
\end{proof}

\subsection{Estimates for the pull-back (proof of Lemma \ref{well-defined})}

We will prove all estimates on the closed ball of radius $n\geq 1$, denoted:
\[K:=\overline{B}_n(0)\] 

First, we estimate the pullback under $\phi_t$ of flat forms. Flatness will be used to increase the degrees of $R$, $g_t$ and $\overline{g}_{t}$ in our estimates. 
\begin{lemma}\label{pullback estimates}
For every $d\in\N$ and $a\in \N^3$ with $|a|=k$, there is a constant $C=C(d,a)$ such that, for any flat form $\alpha\in \Omega^{i}_0(K)$, and any $t\geq 0$:
 \[|D^a(\phi_t^*\alpha)|\leq C\norm{\alpha}^K_{k+d}\, \sigma_{k+i,d}\, (g_t+\overline{g}_t)^{k+d+i},\]
 holds on $K$.
\end{lemma}
\begin{proof}
Step 1: we first prove the estimate for $k=i=0$, i.e.\ for a function $\chi\in C^{\infty}_0(K)$, and for $a=(0,0,0)$. If also $d=0$, the estimate is obvious, since $\sigma_{0,0}=1$. So let $d>0$. Since $\chi$ is flat at $0$, the Taylor formula with integral remainder gives:
\[\chi(v)=\sum_{|b|=d}v^b\int_{0}^1d(1-s)^{d-1}(D^b\chi)(sv)\dif s,\]
where for $v=(x,y,z)$ and $b=(b_1,b_2,b_3)$ we denoted $v^b=x^{b_1}y^{b_2}z^{b_3}$. Thus: 
\[\phi_t^*(\chi)(v)=\sum_{|b|=d}(\phi_t(v))^b\int_{0}^1d(1-s)^{d-1}(D^b\chi)(s\phi_t(v))\dif s.\]
Since $\phi_t(v)=(xg_t,yg_t,z \overline{g}_t)$, we have that
\[|(\phi_t(v))^b|\leq R^d g_t^{b_1+b_2}\, \overline{g}_t^{b_3}\leq R^d(g_t+ \overline{g}_t)^d.\]
On the other hand, since $s\phi_t(v)\in K$, we have that:
\[\Big| \int_{0}^1d(1-s)^{d-1}(D^b\chi)(s\phi_t(v))\dif s\Big|\leq C\norm{\chi}^K_{d}.\]
Using these inequalities and $\sigma_{0,d}=R^d$ we obtain the estimate in this case.

Step 2: we prove now the estimate for a flat function $\chi$, and $a\in \N^3$ with $|a|=k\geq 1$. We use the chain rule to write:
\[D^a(\chi\circ \phi_t)=\sum_{1\leq |b|\leq k}D^b(\chi)\circ \phi_t\ {\sum}'\prod_{j=1}^{b_1}D^{a^{1j}}(x_t)
\prod_{j=1}^{b_2}D^{a^{2j}}(y_t)
\prod_{j=1}^{b_3}D^{a^{3j}}(z_t),
\]
where $b=(b_1,b_2,b_3)$ and $\sum'$ is the sum over all non-trivial decompositions: 
\[a=\sum_{i=1}^3\sum_{j=1}^{b_i}a^{ij}.\]
Since $D^b(\chi)$ is flat at zero, we apply Step 1 with $d\leftarrow k-|b|+d$:
\[|D^b(\chi)\circ \phi_t|\leq C\norm{\chi}^K_{k+d}\ \sigma_{0,k-|b|+d}(g_t+\overline{g}_t)^{k-|b|+d}.\]
Next, by applying Lemma \ref{lemma:bounds on the flow} and  \eqref{submult}, we obtain:
\begin{align*}
\big|\prod_{j=1}^{b_1}D^{a^{1j}}(x_t)
\prod_{j=1}^{b_2}D^{a^{2j}}(y_t)
\prod_{j=1}^{b_3}D^{a^{3j}}(z_t)\big|&\leq C\sigma_{k,|b|-k}g_t^{b_1+b_2}\,\overline{g}_t^{b_3}\\
&\leq C\sigma_{k,|b|-k}(g_t+\overline{g}_t)^{|b|}.
\end{align*}
Using again \eqref{submult}, we obtain the estimate in this case.

Step 3: let $\alpha\in \Omega^i_0(K)$, $i\geq 1$, and $a\in \N^3$. Note that the coefficients of $\phi_t^*(\alpha)$ are sums of elements of the form \[\chi\circ \phi_t\cdot M(\phi_t),\] 
where $\chi\in C^{\infty}_0(K)$ is a coefficient of $\alpha$ and $M(\phi_t)$ denotes the determinant of a minor of rank $i$ of the Jacobian matrix of $\phi_t$. By the Leibniz rule:
\[D^a(\chi\circ \phi_t\cdot M(\phi_t))=\sum_{b+c=a}D^b(\chi\circ \phi_t)D^c(M(\phi_t)).\]
For the first term, we apply Step 2 with $d\leftarrow k-|b|+d=|c|+d$:
\[|D^b(\chi\circ \phi_t)|\leq C\norm{\alpha}^K_{k+d}\sigma_{|b|,|c|+d}(g_t+\overline{g}_t)^{k+d}.\]
Note that $M(\phi_t)$ is a homogeneous polynomial of degree $i$ in the first order partial derivatives of $x_t$, $y_t$ and $z_t$. By Lemma \ref{lemma:bounds on the flow} each such partial derivatives satisfies: 
\[\Big|D^{e}\Big(\frac{\partial u_t}{\partial v}\Big)\Big|\leq C\sigma_{|e|+1,-|e|}(g_t+\overline{g}_t),\]
where $u_t\in \{x_t,y_t,z_t\}$ and $v\in\{x,y,z\}$. Therefore, by applying the Leibniz rule and \eqref{submult}, we obtain:
\[|D^c M(\phi_t)|\leq C \sigma_{|c|+i,-|c|}(g_t+\overline{g}_t)^i.\]
These inequalities imply now the estimates from the statement. 
\end{proof}

Finally, we prove estimates for the derivative of the homotopy operator:

\begin{lemma}\label{derivative of h}
For every $d\in \N$ and $a\in \N^3$ with $|a|=k$, there is a constant $C=C(d,a)$ such that, for any flat form $\alpha\in \Omega^{i}_0(K)$,  and all $t\geq 0$:
 \[|D^a(h'_t(\alpha))|\leq C\norm{\alpha}^K_{k+d}\, \sigma_{k+i-1,d+3}\, g_t\, \overline{g}_t\, (g_t+\overline{g}_t)^{k+d+i},\]
 holds on $K$.
\end{lemma}

\begin{proof}
Recall that
\[h'_t(\alpha)=i_W\phi_t^*(\alpha)=\phi_t^*(i_W\alpha)=\sum_{j=1}^3u_t^j\,\phi_t^*(\alpha_j),\]
where
\[u_t^1=-x_tz_t^2,\ \  \ \ u_t^2=-y_tz_t^2,\ \ \ \  u_t^3=-(x_t^2+y_t^2)z_t\]
and 
\[\alpha_1=i_{\partial_x}\alpha,\ \ \ \ \alpha_2=i_{\partial_y}\alpha,\ \ \ \ \alpha_3=i_{\partial_z}\alpha.\]
First we apply the Leibniz rule:
\[D^a(u_t^j\,\phi_t^*(\alpha_j))=\sum_{b+c=a}D^b(u_t^j) D^c(\phi_t^*(\alpha_j)).\]
By using Lemma \ref{lemma:bounds on the flow}, the Leibniz rule and \eqref{submult}, we obtain:
\[|D^b u_t^j|\leq C \sigma_{|b|,3-|b|}\, g_t\, \overline{g}_t\,(g_t+\overline{g}_t).\]
By applying Lemma \ref{pullback estimates} with $d\leftarrow d+|b|$, we obtain:
\[|D^c(\phi_t^*(\alpha_j))|\leq C\norm{\alpha}^K_{k+d}\,\sigma_{|c|+i-1,d+|b|}\, (g_t+\overline{g}_t)^{k+d+i-1}.\]
These inequalities imply now the estimates from the statement. 
\end{proof}

Finally, we obtain: 
\begin{proof}[Proof of Lemma \ref{well-defined}]
Let $\alpha\in \Omega^i_0(\R^3)$ and $a\in \N^3$, with $|a|=k$. Lemma \ref{derivative of h} with $d=k+i-1$ gives the following on $K$:
\begin{align*}
|D^a(h'_t(\alpha))|&\leq C\norm{\alpha}^K_{2k+i-1}\, \sigma_{k+i-1,k+i+2}\, g_t\, \overline{g}_t\, (g_t+\overline{g}_t)^{2(k+i)-1}\\
&= C\norm{\alpha}^K_{2k+i-1}\, \sigma_{d,d+3}\, g_t\, \overline{g}_t\, (g_t+\overline{g}_t)^{2d+1}.
\end{align*}
Using the definition of the polynomials $\sigma_{k,l}$, that $g_t+\overline{g}_{t}\leq 2$, and Lemma \ref{lemma: finite integrals}, we evaluate the last term for $t>0$: 
\begin{align*}
\sigma_{d,d+3}\, g_t\, \overline{g}_t\, (g_t+\overline{g}_t)^{2d+1}&\leq C\sum_{j=0}^{d}t^jR^{2j+d+3}g_t\, \overline{g}_t\, (g_t+\overline{g}_t)^{2j+1}\leq CR^{d}t^{-\frac{3}{2}}.
\end{align*}
Since $R^d\leq n^d$, we obtain that there exists $C=C(n,k)$ such that for $t>0$:  
\[\norm{h'_t(\alpha)}^K_{k}\leq C \norm{\alpha}^K_{2k+i-1} t^{-\frac{3}{2}}.\]
Thus \eqref{conditions for limit} holds, and therefore $\lim_{t\to\infty}h_t(\alpha)$ exists with respect to the compact-open $C^{\infty}$-topology.
\end{proof}

\bibliography{sl2r_4}
\bibliographystyle{alpha}
\end{document}